\documentclass[12pt,reqno]{amsart}

\usepackage{amsmath,amsthm,amssymb,comment,fullpage}
\usepackage[utf8]{inputenc}
\usepackage{times}
\usepackage[T1]{fontenc}
\usepackage{mathrsfs}
\usepackage{latexsym}
\usepackage[dvips]{graphics}
\usepackage{epsfig}
\usepackage{accents}
\usepackage{float}
\usepackage{lipsum}
\usepackage{amsmath,amsfonts,amsthm,amssymb,amscd}
\input amssym.def
\input amssym.tex
\usepackage{color}
\usepackage{hyperref}
\usepackage{url}
\usepackage{breakurl}
\usepackage{comment}
\usepackage{mathtools}
\usepackage{thmtools}
\usepackage{thm-restate}
\newcommand{\bburl}[1]{\textcolor{blue}{\url{#1}}}
\let\Horig\H

\newcommand{\p}[1]{{\rm Prob}\left(#1\right)}

\numberwithin{equation}{section}

\newtheorem{thm}{Theorem}[section]
\newtheorem{conj}[thm]{Conjecture}
\newtheorem{cor}[thm]{Corollary}
\newtheorem{lem}[thm]{Lemma}

\theoremstyle{plain}

\newtheorem{defn}[thm]{Definition}

\newtheorem{proposition}[thm]{Proposition}

\newtheorem{rek}[thm]{Remark}

\newcommand\be{\begin{equation}}
\newcommand\ee{\end{equation}}
\newcommand\bea{\begin{eqnarray}}
\newcommand\eea{\end{eqnarray}}
\newcommand\bi{\begin{itemize}}
	\newcommand\ei{\end{itemize}}
\newcommand\ben{\begin{enumerate}}
	\newcommand\een{\end{enumerate}}
\newcommand\bc{\begin{center}}
	\newcommand\ec{\end{center}}
\newcommand\ba{\begin{array}}
	\newcommand\ea{\end{array}}

\newcommand{\fof}{\frac{1}{4}}  

\newcommand{\hr}[1]{\href{#1}{\url{#1}}}

\renewcommand{\H}{\mathbb{H}}

\newcommand*{\reff}[1]{\hyperref[#1]{\ref{#1}}}

\makeatletter
\def\@tocline#1#2#3#4#5#6#7{\relax
  \ifnum #1>\c@tocdepth 
  \else
    \par \addpenalty\@secpenalty\addvspace{#2}%
    \begingroup \hyphenpenalty\@M
    \@ifempty{#4}{%
      \@tempdima\csname r@tocindent\number#1\endcsname\relax
    }{%
      \@tempdima#4\relax
    }%
    \parindent\z@ \leftskip#3\relax \advance\leftskip\@tempdima\relax
    \rightskip\@pnumwidth plus4em \parfillskip-\@pnumwidth
    #5\leavevmode\hskip-\@tempdima
      \ifcase #1
       \or\or \hskip 1em \or \hskip 2em \else \hskip 3em \fi%
      #6\nobreak\relax
    \dotfill\hbox to\@pnumwidth{\@tocpagenum{#7}}\par
    \nobreak
    \endgroup
  \fi}
\makeatother
\makeatletter
\newenvironment{subtheorem}[1]{%
  \def\subtheoremcounter{#1}%
  \refstepcounter{#1}%
  \protected@edef\theparentnumber{\csname the#1\endcsname}%
  \setcounter{parentnumber}{\value{#1}}%
  \setcounter{#1}{0}%
  \expandafter\def\csname the#1\endcsname{\theparentnumber.\Alph{#1}}%
  \ignorespaces
}{%
  \setcounter{\subtheoremcounter}{\value{parentnumber}}%
  \ignorespacesafterend
}
\makeatother
\newcounter{parentnumber}

\title{Generalizing Ruth-Aaron Numbers}
\author{Yanan Jiang}
\email{\textcolor{blue}{\href{yanan\_jiang22@milton.edu}{yanan\_jiang22@milton.edu}}}
\address{Milton Academy, Milton, MA 02186}

\author{Steven J. Miller}
\email{\textcolor{blue}{\href{mailto:sjm1@williams.edu, Steven.Miller.MC.96@aya.yale.edu}{sjm1@williams.edu,Steven.Miller.MC.96@aya.yale.edu}}}
\address{Department of Mathematics and Statistics, Williams College, Williamstown, MA 01267}
\let\OLDthebibliography\thebibliography
\renewcommand\thebibliography[1]{
  \OLDthebibliography{#1} 
  \setlength{\parskip}{0pt}
  \setlength{\itemsep}{0pt plus 0.5ex}
}


\date{\today}
\begin{document}
\begin{abstract} 
Let $f(n)$ be the sum of the prime divisors of $n$, counted with multiplicity; thus $f(2020)$ $= f(2^2 \cdot 5 \cdot 101) = 110$. \textit{Ruth-Aaron} numbers, or integers $n$ with $f(n)=f(n+1)$, have been an interest of many number theorists since the famous 1974 baseball game gave them the elegant name after two baseball stars. Many of their properties were first discussed by Erd\Horig{o}s and Pomerance in 1978. In this paper, we generalize their results in two directions: by raising prime factors to a power and allowing a small difference between $f(n)$ and $f(n+1)$. We prove that the number of integers up to $x$ with $f_r(n)=f_r(n+1)$ is $O\left(\frac{x(\log\log x)^3\log\log\log x}{(\log x)^2}\right)$, where $f_r(n)$ is the \textit{Ruth-Aaron} function replacing each prime factor with its $r-$th power. We also prove that the density of $n$ remains $0$ if $|f_r(n)-f_r(n+1)|\leq k(x)$, where $k(x)$ is a function of $x$ with relatively low rate of growth. Moreover, we further the discussion of the infinitude of \textit{Ruth-Aaron} numbers and provide a few possible directions for future study.
\end{abstract}
\maketitle
\tableofcontents
\newpage


\section{Introduction}\label{sec:intro}

On April 8, 1974, Henry Aaron hit his $715$th major league home run, sending him past Babe Ruth, who had a $714$, on the baseball's all-time list. As the event received much advance publicity, the numbers $714$ and $715$ were mentioned by millions of people including mathematicians at the time, whose attention likely deviated from the phenomenal baseball game and was attracted by the beautiful properties of the two consecutive numbers. \\
\indent They first noticed that
\begin{align}
    714\cdot 715\ &=\ 510510\ =\ 2\cdot 3\cdot 5\cdot 7\cdot 11\cdot 13\cdot 17\ =\ P_7\nonumber,
\end{align}
where $P_k$ denotes the product of the first $k$ primes. Without too much effort, we can find expressions for $P_1,P_2,P_3,P_4$ as the product of two consecutive integers. However, after $714$ and $715$, no more products turned up for integer pairs below $10^{6021}$. They thus conjectured that $714$ and $715$ are the largest consecutive integer pair to be written as the product of the first $k$ primes.\\
\indent This conjecture is just the beginning of the beauty of the two integers. In fact, let the unique prime factorization of an integer $n$ be $p_1^{a_1}p_2^{a_2}\cdots p_k^{a_k}$, and define \begin{equation} f(n)\ :=\ a_1p_1+a_2p_2+\cdots +a_kp_k.\end{equation} Nelson, Penney, and Pomerance \cite{NePePo} found that $f(714)=f(715)$. \\
\indent We call the function $f(n)$ the \textit{Ruth-Aaron} function, an integer $n$ with the property $f(n)=f(n+1)$ a \textit{Ruth-Aaron} number, and the pair $n$, $n+1$ a \textit{Ruth-Aaron} pair. A function $f$ is completely additive if $f(ab)=f(a)f(b)$ holds for all integers $a,b$. It easily follows that the \textit{Ruth-Aaron} function has the nice property of being completely additive. A computer search for all the \textit{Ruth-Aaron} numbers not exceeding $50000$ found just $42$ values, suggesting that their density is $0$.
\renewcommand{\arraystretch}{1.2}
\begin{table}[H]
\centering
\begin{tabular}{cccccc}
\hline
    $n$ & $f(n)=f(n+1)$  & $n$  & $f(n)=f(n+1)$  & $n$  & $f(n)=f(n+1)$ \\
    \hline
    5 & 5 & 5405  & 75  & 26642  &193 \\
     8& 6 &  5560 & 150&26649  &66 \\
      15& 8 & 5959  &  160 & 28448  &144 \\
   77   & 18& 6867  & 122 &28809  &117 \\
     125 &15  & 8280 & 40   & 33019  &149 \\
    714   & 29 & 8463  & 54  & 37828  &211 \\
    948   & 86 & 10647  & 39  & 37881  &93 \\
    1330   & 33 & 12351  & 205  &  41261 &64 \\
    1520   & 32 & 14587  & 532  & 42624  &57 \\
    1862   & 35 & 16932  & 107  & 43215  &118 \\
     2491  & 100 & 17080  & 79  &44831   &480 \\
    3248   & 44 & 18490  &  93 & 44891  &82 \\
    4185   & 45 & 20450  &  421 &47544   &299 \\
     4191 & 141 & 24895  & 401  & 49240  &1242 \\
      \hline
\end{tabular}
\caption{\textit{Ruth-Aaron} numbers not exceeding $50,000$.}
\end{table}
In fact, in 1978, only a few years after the famous baseball game, Erd\Horig{o}s and Pomerance \cite{ErPom} proved this result of density $0$. They also established that when $x$ is sufficiently large, the number of \textit{Ruth-Aaron} numbers is at most $C\cdot \frac{x}{(\log x)^{1-\epsilon}}$ for every $0<\epsilon<1$, where $C=C(\epsilon)$ is a constant dependent upon $\epsilon$.\\
\indent In this paper, we extend the results obtained by Erd\Horig{o}s and Pomerance. As an arithmetic function bearing certain similarities to the Sigma function and the Prime Omega function (See Appendix \ref{app:arith}), $f(n)$ renders several natural directions for generalization, one of which is to raise the prime factors to a power. Hence, we first introduce the $r-$th power \textit{Ruth-Aaron} numbers.

\begin{defn}
An integer $n=\prod_{i=1}^kp_i^{a_i}$ is a $r-$th power \textit{Ruth-Aaron} number if
\begin{align}
    f_r(n)\ &=\ f_r(n+1),
\end{align}
where $f_r(n)=\sum_{i=1}^ka_ip_i^r$.
\end{defn}

We prove an upper bound on the number of $r-$th power \textit{Ruth-Aaron} numbers up to $x$ in Section \ref{sec:non} (we will state the result later in the introduction). Our result improved that of Erd\Horig{o}s and Pomerance by a factor of $(\log\log x)^3\log\log\log x/\log x$. Moreover, inspired by Cohen, Cordwell, Epstein, Kwan, Lott, and Miller's study of near perfect numbers \cite{CCEKLM}, we introduce the concept of $k(x)$\textit{-near-Ruth-Aaron} numbers.

\begin{defn}
An integer $n$ is $k(x)$\textit{-near-Ruth-Aaron} when
\begin{align}
    |f_r(n)-f_r(n+1)|\ &\leq\ k(x).
\end{align}
Obviously, when $k(x)=0$, $n$ is a $r-$th power \textit{Ruth-Aaron} number.
\end{defn}

As \textit{Ruth-Aaron} numbers \emph{seem} extremely rare, 
we weaken the condition and investigate how an absolute difference of a small amount between $f_r(n)$ and $f_r(n+1)$ affect the density in Section \ref{sec:rad}; in particular, rather than requiring $f_r(n)$ to equal $f_r(n+1)$, we merely require them to be ``close.'' Moreover, Nelson, Penney, and Pomerance \cite{NePePo} proved the infinitude of \textit{Ruth-Aaron} numbers under Schinzel's Conjecture, which provides us another direction for generalization. In Section \ref{sec:infn}, we prove that there are infinitely many real $r$ such that there is a $r-$th power \textit{Ruth-Aaron} number.\\
\indent As our results are concerning only upper bounds for these numbers, we can and do absorb any set of numbers in the arguments below that is small relative to this bound into our error terms.  \\
\indent For future study, we can place \textit{Ruth-Aaron} numbers in linear equations such as the Fibonacci sequence. We can initiate the study of \textit{Rabonacci} numbers, or integer $n$ with $f(n)=f(n-1)+f(n-2)$. Another possibility is to expand the equation $f(n)=f(n+1)$ to $k-$tuples. Inspired by Erd\Horig{o}s \cite{Er}, we conjecture that for every integer $k\geq 1$, there exists $n,n+1,\dots,n+k$ such that 
\begin{align}\label{ktuple}
    f(n)\ =\ f(n+1)\ &=\ \cdots \ =\ f(n+k).
\end{align}
In fact, a computer search \cite{Nom} tells us that even when $k=2$, solutions are extremely rare. 
\renewcommand{\arraystretch}{1.2}
\begin{table}[H]
\begin{center}
\begin{tabular}{cc}
\hline
  $n$  & $f(n)=f(n+1)=f(n+2)$\\
  \hline
  $417\ 162 $  & $533$\\
 $6\ 913\ 943\ 284$ & $5428$\\
  \hline
\end{tabular}
\caption{List of \textit{Ruth-Aaron} triples below $10^{10}$.}
\end{center}
\end{table}
Due to the rarity of \textit{Ruth-Aaron} triples, following the previous conjecture, we propose that the number of solutions to \eqref{ktuple} is finite for any $k \ge 2$. \\
\indent Although \textit{Ruth-Aaron} numbers are named after two baseball stars (instead of a mathematician like most functions and theorems are) and thus have a more or less recreational origin, their study leads us to a variety of great mathematics, which all play key roles in this paper: 
\begin{itemize}

\item the sieve method \cite{HalRi}, 

\item the Prime Number Theorem \cite{Wei},
 
\item the Chinese Remainder Theorem \cite{DiPeiSal}, 

\item De Bruijn's estimation on integers with regards to the size of their largest prime factors \cite{Bru}, 

\item the Catalan Conjecture \cite{Rib}, 

\item the linear independence of radicals, and 

\item inequalities such as Cauchy-Schwarz and Jensen's. 

\end{itemize}

The \textit{Ruth-Aaron} numbers' profound connection to such mathematics not only once again proves the beauty of their properties, but should also justify why they merit a closer inspection.

\subsection{Notations and Definitions}
We set some notation and define the key objects of this paper.

\begin{defn}
Numbers $f_1,f_2,\dots, f_n$ are linearly independent over $\mathbb{Z}$ if, when coefficients $a_i\in\mathbb{Z}$, the following equation
\begin{align}
    a_1f_1+a_2f_2+\cdots+a_nf_n\ &=\ 0
\end{align}
holds only when $a_i=0$.
\end{defn}

We use the following notations.
\begin{itemize}
    \item We adopt a notation similar to the $\mathcal{A}_{\mathbf{a}}$ used by Luca and Stănică  \cite{LuSta}, to denote linear equations with the \textit{Ruth-Aaron} function.\\
    Let $k\geq 1$ be a positive integer, and let $\mathbf{a}=(a_0,a_1,\dots,a_k)$ be a vector with integer components not all zero. Put $\mathcal{R}^r_{\mathbf{a}}(x)$ for the set of all integers $n$ not exceeding $x$ such that
    \begin{align}
        \sum_{i=0}^ka_if_r(n+i)\ &=\ 0.
    \end{align}
    Then it's not hard to notice that all integers $n$ up to $x$ with $f_r(n)=f_r(n+1)$ coincide with $\mathcal{R}^r_{\mathbf{(1,-1)}}(x)$.
    \item We use $P(n)$ to denote the largest prime factor of integer $n$.
    \item We use $A(x,t)$ to denote the number of $n\leq x$ with $P(n)\geq x^t$, and $a(x,t)$ to denote the fraction of $n\leq x$ with $P(n)\geq x^t$.
    \item We use $\Psi(x,y)$ to denote the number of $n$ up to $x$ with prime factors no greater than $y$.
    \item Big $O$ notation: We write $k(x) = O(g(x))$ or $k(x) \ll g(x)$ if there exists a positive constant $C$ such that $k(x) < C\cdot g(x)$ for all sufficiently large $x$.
    \item Little $o$ notation: We write $k(x) = o(g(x))$ if $\lim_{x\to\infty} f(x)/g(x) = 0$.
    \item We write $k(x)\sim g(x)$ if $\lim_{x\to\infty}k(x)/g(x)=1$.
    \item We denote by $E$ the constant
    \begin{align}
        E\ &=\ \sum_{k=1}^\infty\frac{\left(\mu(k)\right)^2}{k\varphi(k)}\ =\ \frac{\zeta(2)\zeta(3)}{\zeta(6)}\nonumber\\
        &=\ 1.9435964368\dots,
    \end{align}
    where $\zeta(x)$ is the Riemann Zeta function, $\mu(k)$ the M\"obius function, and $\varphi(x)$ the Euler Totient function.
\end{itemize}

\begin{rek}
Our estimations of the number of $n$ at most $x$ that satisfy certain conditions are mostly expressed as $O(g(x))$, where $g(x)$ is a function of $x$. At the expense of tedious constant chasing we could make all the multiplicative constants explicit, but as we are concerned with the decay rate with respect to $x$ there is no need to do so, and we choose big $O$ notation for the sake of readability. Many approximations and scalings presented are not appear optimal, as there is no need since they smaller than our main term.
\end{rek}

\subsection{Main Results}
We obtained the following main results as well as a few others which are not listed below but are introduced later with the proofs of the theorems. These results will be proved using lemmas from Section \ref{sec:pre} and important theorems and conjectures in Appendix \ref{app:important}.

\begin{restatable}[]{thm}{differencethm}\label{thm:ruthaarondifference}
For real $r\geq 1$ and every $\epsilon$ with $0<\epsilon<1$, let $\delta_0=\delta_0(\epsilon)$ be a constant dependent upon $\epsilon$. Let $\delta$ be subject to the following conditions: 
\begin{align}\label{allconditionsdelta}
    \delta\ &\leq\ \delta_0r\epsilon/14\nonumber\\
    0\ <\ \delta\ &<\ \delta^2_0\epsilon/4E^2A\nonumber\\
    \delta\ &<\ \delta_0/4,
\end{align}
where $A$ (see \cite{ErPom}) is a fixed constant around $8$. Then $k(x)$-\textit{near-Ruth-Aaron} numbers have density $0$ for any $k(x)$ such that
\begin{align}
    k(x)\ &\leq\ (x^{r\delta}-x^{-\delta}-1)x^{r/\log\log x}.
\end{align}
\end{restatable}
This result indicates that when $x$ is sufficiently large, if the difference between $f_r(n)$ and $f_r(n+1)$ is essentially less than $x^{\delta'}$, where $\delta$ is arbitrarily small, then the density of $n$ under this condition is still $0$. Hence, not only are \textit{Ruth-Aaron} numbers rare, $k(x)$-\textit{near-Ruth-Aaron} numbers are also rare for $k(x)$ at most a small power of $x$. In particular, if $k(x)$ is a constant or a power of logarithm, $k(x)$\textit{-near-Ruth-Aaron} numbers are likewise very rare.\\
\indent Moreover, recall that $\#\mathcal{R}_{(1,-1)}^r(x)$ denotes the number of integers up to $x$ with $f_r(n)=f_r(n+1)$. We are able to obtain the following new results of $r-$th power \textit{Ruth-Aaron} number:
\begin{subtheorem}{thm}\label{thm:generalizednumberofn}
\begin{restatable}[]{thm}{negativer}
     When $r=-1$,
    \begin{align}
        \#\mathcal{R}^{-1}_{(1,-1)}(x)\ &\ll\ x^{2(\log\log x/\log x)^{1/2}}\ =\ \exp\left(2(\log\log x\log x)^{1/2}\right),
    \end{align}
    which means for every $\epsilon>0$, we can find $x$ sufficiently large such that $\#\mathcal{R}_{(1,-1)}^r(x)\ll x^{\delta}$. 
    In fact, when $r$ is negative 
    \begin{align}
         \#\mathcal{R}^{r}_{(1,-1)}(x)\ &\ll\ x^{O((\log\log x/\log x)^{r/r-1})}.
    \end{align}
    \end{restatable}
    \begin{restatable}[]{thm}{rationalr}
When $r$ is rational but not an integer,
    \begin{align}
        \#\mathcal{R}^r_{(1,-1)}(x)\ &=\ 0.
    \end{align}
    \end{restatable}
    \begin{restatable}[]{thm}{rgeqone}
    When $r\geq 1$ is real,
    \begin{align}
    \#\mathcal{R}^r_{(1,-1)}(x)\ = \ O\left(\frac{x\log\log\log x(\log\log x)^3}{(\log x)^2}\right).
\end{align}
\end{restatable}
\end{subtheorem}
Erd\Horig{o}s and Pomerance \cite{ErPom} conjectured that there are infinitely many \textit{Ruth-Aaron} numbers. While their conjecture is still open, we prove a related problem, namely the infinitude of $r$ for which $\#\mathcal{R}_{(1,-1)}^r(x)>0$.
\begin{restatable}[]{thm}{infinitude}\label{thm:infinitudeofr>0}
There are infinitely many real numbers $r$ such that $\#\mathcal{R}_{(1,-1)}^r(x)>0$.
\end{restatable}

\subsection{Outline}
In Section \ref{sec:pre} we present a few preliminary results that provide a general overview for the problems we study, and which will be used extensively throughout the paper. Then, we discuss the proofs of Theorems \ref{thm:ruthaarondifference}, \ref{thm:generalizednumberofn}, and \ref{thm:infinitudeofr>0} (Sections \ref{sec:rad},    \ref{sec:non}, and \ref{sec:infn}). We conclude with possible future research directions in Section \ref{sec:fut}.

\section{Preliminary Results}\label{sec:pre}
We begin by generalizing some results from Erd\Horig{o}s and Pomerance \cite{ErPom} which will be useful in proving our main theorems. Lemma \ref{lem:atxinfty} and Corollary \ref{cor: epsilonx3} are introduced for the sake of proving Lemma \ref{thm:xdeltapn}, as two cases in terms of the size of $P(n)$ are taken care of by the corollary. Lemma \ref{lem:eupevandu0pv0} and Lemma \ref{lem:plnp} are frequently used in the sieve method to bound various sums over reciprocals of primes. Lemma \ref{lem:generalizedfnpn} and Lemma \ref{lem:lowerboundofp} introduce an upper bound for $f_r(n)$ with regards to $P(n)^r$ and a lower bound for the size of $P(n)$. These results are used extensively throughout the paper. \\ \indent Erd\Horig{o}s and Pomerance \cite{ErPom} first introduced a well-known result due to Dickman \cite{Di} and others which bounds how often the larges prime factor of $n\leq x$ is at least $n^t$. 
\begin{lem}\label{lem:atxinfty}
 For every $x>0$ and every $t$,  $0 \leq t \leq 1$, let $A(x,t)$ denote the number of $n \leq x$ with $P(n) \geq x^t$. The function
\begin{align}\label{ataxt}
    a(t)\ :=\ \lim_{x\to\infty} x^{-1}A(x,t)
\end{align}
is defined and continuous on $[0,1]$.
\end{lem}
\begin{cor}\label{cor: epsilonx3}
From Lemma \ref{lem:atxinfty} we obtain that there exists $\delta_0=\delta_0(\epsilon)$ sufficiently small $(0< \delta_0\leq \fof)$ such that for large $x$, the number of $n\leq x$ with
\begin{align}\label{pnxdelta0}
    P(n)\ <\  x^{\delta_0}\ \textrm{or}\ x^{1/2-\delta_0}\ \leq P(n)\ <\  x^{1/2+\delta_0}
\end{align}
is less than $\epsilon x/3$.
\end{cor}
\begin{proof}
Let $\epsilon>0$. From Lemma \ref{lem:atxinfty}, $a(t)=\lim_{x\to\infty}x^{-1}A(x,t)$, which means the fraction of $n\leq x$ such that $P(n)< x^{\delta_0}$ converges to $a(0)-a(\delta_0)$ when $x\to\infty$. Similarly, the fraction of $n$ which satisfy $x^{1/2-\delta_0}\leq P(n)<x^{1/2+\delta_0}$ converges to $a(1/2-\delta_0)-a(1/2+\delta_0)$.\\ \\
As defined, $a(x,t)$ is the fraction of $n\leq x$ with $P(n)\geq  x^t$. Consider $P(n)\leq x^{\delta_0}$ first. We can find $\delta_1$ and $X_1$ such that  $\forall\ \delta_0\leq \delta_1$ and $x\geq  X_1$, we have $a(x,0)-a(x,\delta_0)\leq \epsilon/8$ and within $\epsilon/2020$ of $a(0)-a(\delta_0)$. For the second condition, because $a(t)$ is continuous, given any $\epsilon$ we can always find $\delta_2$ and $X_2$ such that  if $\delta_0$ is at most $\delta_2$ and $x$ is at least $X_2$ then $a(x,1/2-\delta_0)-a(x,1/2+\delta_0)$ is at most $\epsilon/8$ and within $\epsilon/2020$ of $a(1/2-\delta_0)-a(1/2+\delta_0)$. We take $\delta=\min(\delta_1,\delta_2)$ and $X=\max(X_1,X_2)$, then the fraction of $n\leq x$ satisfying one of the two conditions is no greater than $\epsilon/3$, which means the number of $n$ is no greater than $\epsilon x/3$.
\end{proof}

Lemmas \ref{lem:eupevandu0pv0} and \ref{lem:plnp} are used frequently in later sections to bound various sums over reciprocals of primes.

\begin{lem}\label{lem:eupevandu0pv0}
We have
\begin{eqnarray}\label{eqnarray: eupev}
\sum_{e^u\leq p\leq e^v}\frac{1}{p} & \ < \ & \log(v/u)+\frac{C}{u}\nonumber\\
\sum_{u_0\leq p\leq v_0}\frac{1}{p} & \ < \ & \frac{C + \log(v_0/u_0)}{\log u_0}.  \end{eqnarray}
\end{lem}
\begin{proof}
We use the Abel Partial Summation Formula:
\begin{align}\label{abelsum}
\sum_{1\leq x\leq n}C(x)f(x)\ =\  C(n)f(n)-\int_1^nC(t)f'(t)dt,
\end{align}
and the Prime Number Theorem (weaker version): if $\pi(n)$ is the number of at most $n$, then
\begin{align}\label{pnumthm}
    \pi(n)\ =\ \frac{n}{\log n}+o\left(\frac{n}{(\log n)^2}\right).
\end{align}

First of all, we prove that \begin{align}\label{lnpp}
    \sum_{p\leq n}\frac{\log p}{p}\ =\  \log n+O(1).
\end{align}
We know that the power of $p$ in $n!$ equals $\lfloor\frac{n}{p}\rfloor+\lfloor\frac{n}{p^2}\rfloor+\lfloor\frac{n}{p^3}\rfloor+\cdots$. Thus, we have
\begin{align}
    \log (n!)\ &=\ \sum_{p\leq n}\log p\left(\left\lfloor\frac{n}{p}\right\rfloor+\left\lfloor\frac{n}{p^2}\right\rfloor+\left\lfloor\frac{n}{p^3}\right\rfloor+\cdots\right)\nonumber\\
    &\leq n\cdot\sum_{p\leq n}\frac{\log p}{p}+\sum_{p\leq n}\log p\left(\frac{n}{p^2}+\frac{n}{p^3}+\cdots\right)\nonumber\\
    &=\ n\cdot\sum_{p\leq n}\frac{\log p}{p}+n\sum_{p\leq n}\frac{\log p}{p(p-1)}.
\end{align}

Since $\sum_{p\leq n} \frac{\log p}{p(p-1)}< \sum_{k=1}^n \frac{\log k}{k(k-1)}$ converges, we have $n\sum_{p\leq n}\frac{\log p}{p(p-1)} = O(n)$. Therefore,
\begin{align}
  \log(n!)\ =\ n\cdot\sum_{p\leq n}\frac{\log p}{p}+O(n).
\end{align}
On the other hand,
\begin{align}
    \log(n!)\ &=\ \sum_{k=1}^n\log k\ =\ \int_1^n\log x \ dx +O(\log n)\nonumber\\
    &=\ n\log n-n+O(\log n)\nonumber\\
    &=\ n\log n+O(n).
\end{align}

Comparing the two results, we obtain \eqref{lnpp}:
\begin{align}
\sum_{p\leq n}\frac{\log p}{p}\ =\ \log n+O(1).
\end{align}

We now use the Abel Partial Summation Formula. Let $f(x) =\frac{1}{\log x}$ and $C(x) =\sum_{p\leq n}\frac{\log p}{p}$, then we have
\begin{align}
    \sum_{p\leq n}\frac{1}{p}\ &=\ \frac{C(n)}{\log n}+\int_1^n\frac{C(t)}{t(\log t)^2}dt\nonumber\\
    &=\ 1+\frac{O(1)}{\log n}+\int_1^n\frac{1}{t\log t}dt+\int_1^n\frac{O(1)}{t(\log t)^2}dt\nonumber\\
    &=\ \log\log n+O\left(\frac{1}{\log n}\right)+1.
\end{align}

Then
\begin{align}
   \sum_{e^u\leq p\leq e^v}\frac{1}{p}\ &=\  \sum_{p\leq e^v}\frac{1}{p}-\sum_{p\leq e^u}\frac{1}{p}+\frac{1}{e^u}\nonumber\\
   &=\ \log(v/u)+O\left(\frac{1}{u}\right)\nonumber\\
   &\leq\ \log (v/u)+\frac{O(1)}{u}.
\end{align}
Therefore, we have \begin{align}\label{eupevstrong}
\sum_{e^u\leq p\leq e^v}\frac{1}{p}\ \leq\ \frac{O(1)+u\log(v/u)}{u}.
\end{align}
Let $u_0=e^u, v_0=e^v$, then
\begin{align}\label{eq:u0pv0strong}
    \sum_{u_0\leq p\leq v_0}\frac{1}{p}\ \leq\  \log\left(\frac{\log v_0}{\log u_0}\right)+\frac{C}{\log u_0}.
\end{align}

We now prove that $\log\left(\frac{\log v_0}{\log u_0}\right)+\frac{C}{\log u_0}< \frac{C+\log (v_0/u_0)}{\log u_0}$. Since we have $t<e^{t-1}$ for any $t>1$,
\begin{align}
   \frac{v}{u}\ &<\ e^{v/u-1}
    \nonumber\\
 \log (v/u)\ &<\ \frac{v}{u}-1
    \nonumber\\
  u\log(v/u)\ &<\ (v-u)\nonumber\\
    \log u_0\log\left(\frac{\log v_0}{\log u_0}\right)\ &<\ \log(v_0/u_0).
\end{align}

Therefore, we have
\begin{align}\label{eq:u0pv0weak}
\sum_{u_0\leq p\leq v_0}\frac{1}{p}\ <\  \frac{C+\log(v_0/u_0)}{\log u_0}.
\end{align}
\end{proof}
\begin{rek}
We have two inequalities in Lemma \ref{lem:eupevandu0pv0}; \eqref{eq:u0pv0weak} is weaker than \eqref{eupevstrong} as we replace $u/v$ by $e^{u/v}-1$, and the following theorems apply mostly the result from \eqref{eq:u0pv0weak}. In fact, it turns out that in this paper the inequality in \eqref{eq:u0pv0weak} is tight enough and applicable in most cases, and we will adopt the inequality in \eqref{eupevstrong} otherwise.
\end{rek}
\begin{lem}\label{lem:plnp}
We have \begin{eqnarray}
\sum_{p\geq t}\frac{1}{p\log p}\ =\ \frac{1}{\log t}+O\left(\frac{1}{(\log t)^2}\right).
\end{eqnarray}
\end{lem}
\begin{proof}
We use Abel's Partial Summation Formula. Let

\begin{align}\label{staxft}
    S(t)\ &=\ \sum_{p\geq t}\frac{1}{p\log p}=\ \sum_{p\geq t}\frac{\log p}{p}\cdot(\log p)^{-2}.\nonumber\\
   A(x)\ &=\sum_{p\leq x}\frac{\log p}{p}-\sum_{p\leq t}\frac{\log p}{p}=\ \log x-\log t+O(1).\nonumber\\
   f(t)\ &=\ \frac{1}{(\log t)^2},\
   f'(t)\ =\ -\frac{2}{t(\log t)^3}.
\end{align}
    Because $A(x)$ concerns primes in the interval $(t,x]$, the following sum differs from the one in \eqref{staxft} by at most the term $\frac{1}{t\log t}$, which can be absorbed in the error; thus, it is sufficient to study $S(t)$:
\begin{align}\label{soft}
    S(t)\ &=\ \frac{1}{t\log t}+ A(\infty)f(\infty)-\int_{t}^\infty A(x)f'(x)\ dx\nonumber\\
    &=\ \frac{1}{t\log t}+2\int_{t}^\infty(\log x-\log t+O(1))\frac{1}{x(\log x)^3}\ dx \nonumber\\
    &=\ \frac{1}{t\log t}+2\int_t^\infty\frac{1}{x(\log x)^2}\ dx-2\int_t^\infty(\log t-O(1))\frac{1}{x(\log x)^3}\ dx\nonumber\\
    &=\ \frac{1}{t\log t}- 2\frac{1}{\log x}\bigg|_t^\infty+\frac{\log t}{(\log x)^2}\bigg|_t^\infty-\frac{O(1)}{(\log x)^2}\bigg|_t^\infty\nonumber\\
    &=\ \frac{2}{\log t}-\frac{1}{\log t}+O\left(\frac{1}{(\log t)^2}\right)\nonumber\\
    &=\ \frac{1}{\log t}+O\left(\frac{1}{(\log t)^2}\right),
\end{align}
which completes our proof.
\end{proof}
\begin{lem}\label{lem:generalizedfnpn}
If $P(n)\geq 5$ and $r\geq 1$, we have
\begin{align}\label{frnpnr}
    f_r(n)\ &\leq\ P(n)^r\cdot\frac{\log n}{\log P(n)}.
\end{align}
\end{lem}
\begin{proof}
Consider the function $g(x)=\frac{x^r}{\log x}$, where $r$ is a real number and $x\geq e^{1/r}$, then
\begin{align}
    g'(x)\ &=\ \frac{rx^{r-1}\log x-x^r\cdot 1/x}{(\log x)^2}\nonumber\\
    &=\ \frac{x^{r-1}(r\log x-1)}{(\log x)^2}\nonumber\\
    &>\ 0,
\end{align}
which means $g(x)$ increases when $x\geq e^{1/r}$. Without the loss of generality, let $p_1=P(n)$, then we have
\begin{align}
    f_r(n)\ &=\ \sum_{i=1}^ka_ip_i^r\nonumber\\
    &\leq\ \sum_{i=1}^ka_ip_1^r\cdot\frac{\log p_i}{\log p_1}\nonumber\\
    &=\ \frac{P(n)^r}{\log P(n)}\cdot\sum_{i=1}^k\log p_i^{a_i}\nonumber\\
    &=\ P(n)^r\cdot\frac{\log n}{\log P(n)},
\end{align}
which completes our proof.
\end{proof}
\begin{lem}\label{lem:lowerboundofp}
The number of $n$ up to $x$ not satisfying
\begin{align}
    P(n)\ &>\ x^{1/\log\log x}\ \textrm{and } P(n+1)\ >\ x^{1/\log\log x}
\end{align}
is at most $O\left(\frac{x}{(\log x)^2}\right)$.
\end{lem}
\begin{proof} Let $\Psi(x,y)$ denote the number of  $n$ up to $x$ with prime factors no greater than $y$, and set $u = \log x / \log y$.
 A result from De Bruijn \cite{Bru} states that if $(\log x)^2<y\leq x^{1/3}$ then
\begin{align}
    \log\Psi(x,y)\ &<\ x(\log y)^2\exp\left(-u\log u-u\log\log u+O(u)\right).
\end{align}
We replace $y$ with $x^{1/\log\log x}$ and find
\begin{align}
    \Psi(x,y)\ &<\ x\left(\frac{\log x}{\log\log x}\right)^2\cdot\exp\left(O(\log\log x)-\log\log x(\log\log\log x+\log\log\log\log x)\right)\nonumber\\
    &<\ x\left(\frac{\log x}{\log\log x}\right)^2\cdot\exp\left(O(-\log\log x\log\log\log x)\right)\nonumber\\
    &<\ O\left(x\left(\frac{\log x}{\log\log x}\right)^2\cdot\frac{1}{(\log x)^{\log\log\log x}}\right)\nonumber\\
    &<\ O\left(\frac{x}{(\log x)^2}\right).
\end{align}
Therefore, the number of $n$ for which
\begin{align}\label{loglogx}
    p\ >\ x^{1/\log\log x}\textrm{ and } q\ >\ x^{1/\log\log x}
\end{align}
doesn't hold is $O(x/(\log x)^2)$.
\end{proof}

\begin{rek}
We hence know that for the majority of integers no greater than $x$, $P(n) > x^{1/\log\log x}$, which means $P(n)$ is typically larger than $\log x$ to any power. Moreover, De Koninck and Ivić \cite{KonIv} showed that the sum of largest prime factors of integers up to $x$ is of size $x^2/\log x$, which means the average largest prime factor of integers up to $x$ is of size $x/\log x$. When $x$ is sufficiently large, $x/\log x > x^{\delta}$ for any $0<\delta<1$. Therefore, the average largest prime factor is greater than $x$ to any power less than $1$, indicating that a considerable number of $P(n)$ are very large.
\end{rek}

\section{Proof of Theorem \ref{thm:ruthaarondifference}}\label{sec:rad}
In this section we prove Theorem \ref{thm:ruthaarondifference}. We first introduce the following two lemmas generalized from Erd\Horig{o}s and Pomerance \cite{ErPom}, the first of which indicates that the largest prime factors of two consecutive integers are usually far apart, and the second proves that $P(n)^r$ is often the dominating element that determines the size of $f_r(n)$.
\begin{lem}\label{thm:xdeltapn}
For each $0< \epsilon<1$, there is a $\delta>0$ such that for sufficiently large $x$, the number of $n\leq x$ with
\begin{align}\label{deltapn}
    x^{-\delta}\ <\  P(n)/P(n+1)\ <\  x^{\delta}
\end{align}
is less than $\epsilon x$.
\end{lem}
\begin{proof}
We know from Corollary \ref{cor: epsilonx3} that $\exists\ \delta_0=\delta_0(\epsilon)$ sufficiently small $(0< \delta_0\leq 1/4)$ such that for large $x$, the number of $n\leq x$ with
\begin{align}\label{pnxdelta0}
    P(n)\ <\  x^{\delta_0}\ \textrm{or}\ x^{1/2-\delta_0}\ \leq P(n)\ <\  x^{1/2+\delta_0}
\end{align}
is less than $\epsilon x/3$. Now we consider the remaining cases:
\begin{align*}
 &\textrm{(i) }   x^{\delta_0}\ \leq\ P(n)\ <\  x^{1/2-\delta_0}\\
 &\textrm{(ii) } P(n)\ \geq \ x^{1/2+\delta_0}.
\end{align*}

We will show that for every $0<\epsilon<1$, there exists $\delta$ such that such that the number of $n\leq x$ satisfying one of (i) and (ii) while \eqref{deltapn} holds is less than $\epsilon x/3$.  \ \\

\noindent\emph{We consider (i) first.} We know that for each pair of primes $p,q$, there are at most $1+\lfloor\frac{x}{pq}\rfloor$ choices\footnote{This is because the number of $n\leq x$ such that  $P(n)=p,P(n+1)=q$ is bounded by the number of $n\leq  x$ such that  $n\ \equiv\ 0\ (\textrm{mod p})$ and $n\ \equiv\ -1\ (\textrm{mod q})$. By Chinese Remainder Theorem, $n\ \equiv\ bp\ (\textrm{mod }pq)$ which means there are at most $1+\left\lfloor\frac{x}{pq}\right\rfloor$ such $n\leq x$. } of $n\leq x$ for which $P(n)=p,\ P(n+1)=q$. For inequality \eqref{deltapn} to hold, $px^{-\delta}< q< px^{\delta}$. Then for large $x$, the number of $n\leq  x$ in case \eqref{deltapn} holds is (we many assume $0< \delta< \delta_0/4$)
\begin{align}\label{1xpq}
    \sum_{\substack{x^{\delta_0}\leq p < x^{1/2-\delta_0}\\px^{-\delta}<q<px^{\delta}}}\left(1+\left\lfloor\frac{x}{pq}\right\rfloor\right)\ &<\ x^{1-2\delta_0+\delta}+x\sum_{x^{\delta_0}\leq p<x^{1/2-\delta_0}}\frac{1}{p}\sum_{px^{-\delta}<q<px^{\delta}}\frac{1}{q}\nonumber\\
    \ &<\ x^{1-2\delta_0+\delta}+x\sum_{x^{\delta_0}\leq p<x^{1/2-\delta_0}}\frac{1}{p
    }\cdot\frac{C+\log (px^\delta/ px^{-\delta})}{\log(px^{-\delta})}\textrm{ (Lemma \ref{lem:eupevandu0pv0})}\nonumber\\
    &<\ x^{1-2\delta_0+\delta}+x\sum_{x^{\delta_0}\leq p<x^{1/2-\delta_0}}\frac{1}{p}\cdot\frac{C+\log x^{2\delta}}{\log(px^{-\delta})}.
\end{align}
Moreover, since $p\geq  x^{\delta_0}>x^{4\delta}>x^{3\delta}$, we have $x^{-\delta}>p^{-1/3}$,
\begin{align}\label{clogx2deltap1}
x\sum_{x^{\delta_0}\leq p<x^{1/2-\delta_0}}\frac{1}{p}\cdot\frac{C+\log(x^{2\delta})}{\log (px^{-\delta})}\ &<\ x\sum_{x^{\delta_0}\leq p<x^{1/2-\delta_0}}\frac{1}{p}\cdot\frac{C+\log(x^{2\delta})}{\log (p^{2/3})}\nonumber\\
&=\ x\sum_{x^{\delta_0}\leq p<x^{1/2-\delta_0}}\frac{1}{p}\cdot\frac{\frac{3}{2}C+\log(x^{3\delta})}{\log p}.
\end{align}
When $x$ is sufficiently large, we have
\begin{align}\label{clogx2deltap2}
    x\sum_{x^{\delta_0}\leq p<x^{1/2-\delta_0}}\frac{1}{p}\cdot\frac{C+\log(x^{2\delta})}{\log (px^{-\delta})}\ &<\ x\sum_{x^{\delta_0}\leq p<x^{1/2-\delta_0}}\frac{1}{p}\cdot\frac{1.1\log(x^{3\delta})}{\log p}\nonumber\\
    &<\ \frac{3.3}{3}\cdot3\delta x\log x\sum_{x^{\delta_0}\leq p<x^{1/2-\delta_0}}\frac{1}{p\log p}\nonumber\\
    &<\ \frac{4}{3}\cdot 3\delta x\log x\frac{1}{\delta_0\log x}\textrm{ (Lemma \ref{lem:plnp})}\nonumber\\
    &<\ 4x\delta/\delta_0.
\end{align}
Therefore, we have
\begin{align}\label{4deltax}
     \sum_{\substack{x^{\delta_0}\leq p<x^{1/2-\delta_0}\\px^{-\delta}<q<px^{\delta}}}\left(1+\left\lfloor\frac{x}{pq}\right\rfloor\right)\ &<\ x^{1-2\delta_0+\delta}+4\delta x/\delta_0
\end{align}
If we choose $\delta$ such that
\begin{align}\label{epsilon13}
  \delta<\delta_0\epsilon/13,
\end{align}
then \eqref{4deltax} implies there are less than $\epsilon x/3$ choices of $n$. \\ \\

\noindent \emph{Now we consider case (ii).} Let $a=n/P(n)$ and $b=(n+1)/P(n+1)$. Then $a< x/x^{1/2+\delta_0}=x^{1/2-\delta_0}$, and because \eqref{thm:xdeltapn} holds, $b\leq \lfloor x^{1/2-\delta_0+\delta}\rfloor+1$ ($b=x^{1/2-\delta_0+\delta}$ is possible only when $x^{1/2-\delta_0+\delta}\in\mathbb{Z}$)\footnote{This is because $b=\frac{n+1}{P(n+1)}< \frac{x+1}{P(n)}\cdot x^{\delta}< \frac{x+1}{x^{1/2+\delta_0}}\cdot x^{\delta}=\frac{x+1}{x}\cdot x^{1/2-\delta_0+\delta}$, meaning that $b\leq  \lfloor\frac{x+1}{x}x^{1/2-\delta_0+\delta}\rfloor=\lfloor x^{1/2-\delta_0+\delta}+x^{-1/2-\delta_0+\delta}\rfloor$.
Because $-1/2-\delta_0+\delta< -1/2$, when $x$ is sufficiently large, we have $\lfloor x^{1/2-\delta_0+\delta}+x^{-1/2-\delta_0+\delta}\rfloor\leq \lfloor x^{1/2-\delta_0+\delta}\rfloor+1$; therefore, $b\leq \lfloor x^{1/2-\delta_0+\delta}\rfloor+1$.}, and $x^{-\delta}/2< a/b=nP(n+1)/(n+1)P(n)< 2x^\delta$. Meanwhile, when $a,b$ are fixed, the number of $n\leq x$ for which $n=aP(n),\ n+1=bP(n+1)$ is at most the number of primes $p\leq x/a$ such that  $(ap+1)/b$ is a prime. B\'ezout's Identity (or the Euclidean algorithm) gives that for integers $a,b$, there always exists $m,n\in\mathbb{Z}$ such that $ma+nb=\gcd(a,b)$. Now, because $\gcd(a,b)=1$ and $2\ |\ ab$, the number of $p$ with $bq-ap=1$ is greater than 1. Moreover, given $(ap+1)/b$ is an integer, all $p$ are in the same residue class modulo $b$. Let $p=kb+c$, where $k,c\in\mathbb{Z_+}$ and $c\in[0,b-1]$. Let $d=(ac+1)/b$. Then we are counting positive integer $k$ not exceeding $x/ab$ with primes $kb+c=P(n)$, $ka+d=P(n+1)$. Let the number of such $k$ be ${\rm Pairs}(x)$. By Brun's Method \cite{HalRi}, the number of primes $p\leq x$ congruent to $t$ modulo $q$ is no greater than $\frac{Ax}{q\log(x/q)}\cdot\prod_{p|q}(1-p^{-1})^{-1}$. Applying this result, we have
\begin{align}
   {\rm Pairs}(x)\ &\leq\ \frac{Ax}{ab\log^2(x/ab)}\prod_{p\ |\ ab}\left(1-\frac{1}{p}\right)^{-1}\nonumber\\
   &=\ \frac{Ax}{\varphi(a)\varphi(b)\log^2(x/ab)},
\end{align}
where $A$ is a constant of size around $8$ and $\varphi$ is Euler's totient function, or the number of integers up to $n$ that are relatively prime to $n$. Because we are investigating only the $x-$dependent components and not the multiplicative constants, our only concern here is the size of $1/\log^2(x/ab)$ in relation to the change of $a,b$. In particular, as all summations are positive, no cancellation is involved, and thus it suffices to show this sum is of the same size for all $a, b$ in our ranges.

We now bound above and below $1/\log^2(x/ab)$. Because $a,b\in\mathbb{N}$,
\begin{align}
    \frac{1}{\log^2(x/ab)}\ >\ \frac{1}{(\log x)^2}
\end{align}
Meanwhile,
\begin{align}\label{1xab}
    \frac{1}{\log^2(x/ab)}\ & <\ \frac{1}{\log^2(x/(x^{2\cdot{(1/2-\delta_0})}\cdot 2x^\delta))}\nonumber\\
&=\ \frac{1}{\log^2(2x^{2\delta_0-\delta})}\nonumber\\
    &\ <\frac{2}{(2\delta_0-\delta)^2(\log x)^2}.
\end{align}
This shows us that we can remove $ab$ at a cost of a multiplicative change in the result. We now use the result of Landau \cite{Lan}: if $E=\zeta(2)\zeta(3)/\zeta(6)$, then
$\sum_{n\leq x}1/\varphi(n)=E\log x+o(1)$. Therefore, using \eqref{1xab},
\begin{align}\label{rmPairsx}
   {\rm Pairs}(x)\ &<\ \frac{2Ax}{(2\delta_0-\delta)^2(\log x)^2}\sum_{a\leq x^{1/2-\delta_0}}\frac{1}{\varphi(a)}\sum_{ax^{-\delta}/2<b<2ax^{\delta}}\frac{1}{\varphi(b)}\nonumber\\
   &< \ \frac{2Ax}{(2\delta_0-\delta)^2(\log x)^2}\sum_{a\leq x^{1/2-\delta_0}}\frac{E\log( 2ax^{\delta}/\frac{ax^{-\delta}}{2})+o(1)}{\varphi(a)} \nonumber\\
   &<\ \frac{Ax}{(2\delta_0-\delta)^2(\log x)^2}\sum_{a\leq x^{1/2-\delta_0}}\frac{3E\log x^{2\delta}+o(1)}{\varphi(a)}\nonumber\\
   &<\ \frac{7EA\delta x}{(2\delta_0-\delta)^2\log x}\sum_{a\leq x^{1/2-\delta_0}}\frac{1}{\varphi(a)}\nonumber\\
   &<\ \frac{8E^2A\delta x(1/2-\delta_0)\log x}{(2\delta_0-\delta)^2\log x}\nonumber\\
   &<\ \frac{4E^2A\delta x}{(2\delta_0-\delta)^2}.
\end{align}
Let
\begin{align}\label{d0e4e2a}
    0\ <\ \delta\ <\ \delta^2_0\epsilon/4E^2A \textrm{ and }\delta\ <\ \delta_0/4,
\end{align}
then \eqref{rmPairsx} implies there are fewer than $\epsilon x/3$ choices for such $n$. Thus, if we choose $\delta$ such that  both \eqref{epsilon13} and \eqref{d0e4e2a} hold, then there are less than $\epsilon x$ choices of $n$ for every sufficiently large $x$, completing our proof.
\end{proof}
\begin{rek}
For our purposes, the estimation in inequalities \eqref{1xpq} and \eqref{clogx2deltap1} is sufficient, but if we substitute $\sum_{px^{-\delta}<q<px^\delta}1/q$ with $\log\frac{\log px^\delta}{\log px^{-\delta}}+C/\log px^{-\delta}$, in other words, if we use inequality \eqref{eq:u0pv0strong} rather than \eqref{eq:u0pv0weak}, then with a bit more work we could get $\delta<\frac{\epsilon\delta_0}{6.12\log(1/2\delta_0)}$.
\end{rek}
\begin{rek} Moreover, we can easily extend the result to $r\geq 1$. From \eqref{epsilon13} and \eqref{d0e4e2a} we know $\delta$ depends on $\epsilon$, and because $\epsilon$ is arbitrary between $0$ and $1$, $\delta$ can be very small. For every $0<\epsilon<1$, we find $\delta'=r\cdot \delta$ (where $\delta$ satisfies \eqref{epsilon13} and \eqref{d0e4e2a}; hence $\delta'$ can be very small) such that for sufficiently large $x$, the number of $n\leq x$ with
\begin{align}
    x^{-\delta'}\ <\ P(n)^r/P(n+1)^r\ <\ x^{\delta'}
\end{align}
is less than $\epsilon x$.
\end{rek}
\begin{lem}\label{thm:1minusepsilon}
When $r\geq 1$, for every $\epsilon>0$, let $\delta=\delta_0r\epsilon/14$, then for sufficiently large $x$ there are at least $(1-\epsilon)x$ choices for composite integer $n\leq x$ such that
\begin{align}\label{1plusxne}
    P(n)^r\ \leq\  f_r(n)\ <\  (1+x^{-\delta})P(n)^r.
\end{align}
\end{lem}
\begin{proof}
We know that any $n$ at most $x$ is divisible by at most $\log x/\log 2$ primes. By the Prime Number Theorem, the number of primes up to $x$ is $O(x/\log x)$. Then for any $0<\epsilon_0<1$, we can always find sufficiently large $x$ such that $O(x/\log x)=\epsilon_0 x$, which means the number of prime $n$ up to $x$ is $o(x)$ and can be absorbed. Thus, we have for sufficiently large $x$ and composite $n$:
\begin{align}\label{frn}
    f_r(n)\ &\ =\ P(n)^r + f_r(n/P(n)\nonumber\\
    &\ \leq\ P(n)^r+P(n/P(n))^r\cdot\log x/\log 2\nonumber\\
    &\ <\ P(n)^r+P(n/P(n))^r\cdot x^{\delta}
\end{align}
for any fixed $\delta$. We prove that there are at most $\epsilon x$ choices of $n\leq x$ such that
\begin{align}\label{frngeq}
f_r(n)\ \geq\ (1+x^{-\delta})P(n)^r
\end{align}
holds for all except $o(x)$ choices of $n\leq x$. Then, for such $n$, if \eqref{frngeq} holds, from \eqref{frn} we have
\begin{align}\label{p(n/pn)r}
    P(n/P(n))^r\ >\ x^{-2\delta}\cdot P(n)^r.
\end{align}
Let $\epsilon>0$. We know from Corollary \ref{cor: epsilonx3} that there exists $\delta_0=\delta_0(\epsilon)$ such that for large $x$ the number of $n\leq x$ with $P(n)<x^{\delta_0}$ is at most $\epsilon x/3$. Meanwhile, we know for each pair of primes $p,q$, there are at most $\lfloor\frac{x}{pq}\rfloor$ choices of $n\leq x$ with $P(n)^r=p^r$ and $P(n/P(n))^r=q^r$. Hence, from \eqref{p(n/pn)r}, for large $x$ the number of $n\leq x$ for which \eqref{p(n/pn)r} doesn't hold is at most (assume $0<\delta<\delta_0/3$):
\begin{align}\label{oxepsilonx3generalized}
    o(x)+\frac{\epsilon x}{3}+\sum_{\substack{x^{\delta_0}<p\\x^{-2\delta/r}p<q\leq p}}\left\lfloor\frac{x}{pq}\right\rfloor\ &<\ \frac{\epsilon x}{2}+x\sum_{x^{\delta_0}<p}\frac{1}{p}\sum_{x^{-2\delta/r}p<q\leq p}\frac{1}{q}\nonumber\\
    &<\ \frac{\epsilon x}{2}+x\sum_{x^{\delta_0}<p}\frac{1}{p}\cdot\frac{C+(2\delta\log x)/r}{\log (x^{-2\delta/r}p)}\nonumber\\
    &<\ \frac{\epsilon x}{2}+\frac{6\delta}{r}\cdot x\log x\sum_{x^{\delta_0}<p}\frac{1}{p\log p}\nonumber\\
      &<\ \frac{\epsilon x}{2}+\frac{7\delta x}{r\delta_0}\ (\textrm{Lemma \ref{lem:plnp}})
\end{align}
where $o(x)$ accounts for all $n\leq x$ of the form $m=P(n)^k$ for some positive integer $k$. We take $\delta=\delta_0r\epsilon/14$, then \eqref{oxepsilonx3generalized} is no greater than $\epsilon x$. Therefore, the number of $n\leq x$ such that  $P(n)^r <  f_r(n) < (1+x^{-\delta}P(n)^r)$ is at least $(1-\epsilon)x$, completing the proof.
\end{proof}

We recall Theorem \ref{thm:ruthaarondifference}.
\differencethm*
Now, we know from Lemma \ref{thm:xdeltapn} that there are less than $\epsilon x$ choices of $n$ which satisfy
\begin{align}
 x^{-r\delta} \ < \ \frac{P(n)^r}{P(n+1)^r}\ <\ x^{r\delta}.
\end{align}
Without loss of generality, let $P(n)>P(n+1)$; the other case is handled similarly\footnote{One of the Erd\Horig{o}s-Turán conjectures asserts that the asymptotic density of  $n\leq x$ with $P(n)>P(n+1)$ is $\frac{1}{2}$ \cite{ErPom}. Erd\Horig{o}s and Pomerance \cite{ErPom} showed that the number of $n$ up to $x$ with $P(n)>P(n+1)$ is greater than $0.0099 x$. This result was recently improved by Lü and Wang \cite{LuWang}, who proved that the density is larger than $0.2017$.}. Then there are less than $\epsilon x$ choices of $n$ with
\begin{align}\label{rdelta}
    (x^{-r\delta}-1)P(n+1)^r\ &<\ P(n)^r-P(n+1)^r\ <\ (x^{r\delta}-1)P(n+1)^r.
    \end{align}
Because $r\delta>0$, the LHS of \eqref{rdelta} is negative, which means there are less than $\epsilon x$ choices of $n$ with
    \begin{align}
    0\ &<\ P(n)^r-P(n+1)^r\ <\ (x^{r\delta}-1)P(n+1)^r.
\end{align}
Then for at least $(1-\epsilon) x $ choices of $n$, we have
\begin{align}
P(n)^r-P(n+1)^r\ &>\ (x^{r\delta}-1)P(n+1)^r.
\end{align}
Meanwhile, we know from Lemma \ref{thm:1minusepsilon} that for all but $\epsilon x$ choices of $n$ we have
\begin{align}
    P(n)^r\ &<\ f_r(n)\ <\ (1+x^{-\delta})P(n)^r\nonumber\\
      P(n+1)^r\ &<\ f_r(n+1)\ <\ (1+x^{-\delta})P(n+1)^r.
\end{align}
Therefore, there are more than $(1-\epsilon)x$ choices of $n$ with
\begin{align}
f_r(n)-f_r(n+1)\ &>\ P(n)^r-(1+x^{-\delta})P(n+1)^r\nonumber\\
&>\ (x^{r\delta}-x^{-\delta}-1)P(n+1)^r\nonumber\\
&>\ (x^{r\delta}-x^{-\delta}-1)x^{r/\log\log x},
\end{align}
which means the density of $n$ with
\begin{align}\label{absfnfn+1}
    |f_r(n)-f_r(n+1)|\ &<\ (x^{r\delta}-x^{-\delta}-1)x^{r/\log\log x}
\end{align}
is $0$. In fact, when $x$ is sufficiently large, the RHS of \eqref{absfnfn+1} is greater than $x^\delta$, $(\log x)^k$, for any $k$, or $O(1)$, which means the density of $n$ up to $x$ with $|f_r(n)-f_r(n+1)|<k(x)$, where $k(x)$ is one of the functions above, is also $0$.

\section{Proof of Theorem \ref{thm:generalizednumberofn}}\label{sec:non}
Recall that the notation $\#\mathcal{R}_{(1,-1)}^r(x)$ denotes the number of integers up to $x$ with $f_r(n)=f_r(n+1)$. Recall Theorem \ref{thm:generalizednumberofn}.
\negativer*
\rationalr*
\rgeqone*
We first show that the number of $n\leq x$ when $r=-1$ is less than $x^{2(\log\log x/\log x)^{1/2}}$, which means for a fixed $\delta$ arbitrarily small and $x$ sufficiently larger, $\#\mathcal{R}_{(1,-1)}^{-1}(x)\leq x^\delta$. Then, we will show, using linear independence, that when $r$ is a non-integer rational, $r-$th power \textit{Ruth-Aaron} numbers do not exist. Last, we will present an initial result by Erd\Horig{o}s and Pomerance \cite{ErPom} regarding the number of \textit{Ruth-Aaron} numbers before generalizing it to $r\geq 1$.

\subsection{Negative $r$}
\negativer*
\begin{proof}
First, we prove that when $r=-1$, $r-$th power \textit{Ruth-Aaron} number $n$ exists only when, in the unique prime factorization of $n$ and $n+1$, the power of each prime divides the prime itself. In other words,
\begin{align}
    n\ &=\ p_1^{a_1}\cdot p_2^{a_2}\cdots p_k^{a_k}\nonumber\\
    n+1\ &=\ q_1^{b_1}\cdot q_2^{b_2}\cdots q_l^{b_l},
\end{align}
where $p_i|a_i$ and $q_i|b_i$. \\
\indent Let $a_i',b_i'\in\mathbb{Q}$, and $a_i'=\frac{a_i}{p_i}$, $b_i'=\frac{b_i}{q_i}$, where $a_i,b_i\in\mathbb{N}$. Then we have
\begin{align}\label{proveaidividespi}
 a_1'+a_2'+\cdots+a_k' \ = \ f_{-1}(n)\ &=\ f_{-1}(n+1)\ =\ b_1'+b_2'+\cdots+b_l'\nonumber\\
 \frac{a_1}{p_1}+\frac{a_2}{p_2}+\cdots+\frac{a_k}{p_k}\ &=\ \frac{b_1}{q_1}+\frac{b_2}{q_2}+\cdots\frac{b_l}{q_l'}\nonumber\\
 ((a_1p_2\cdots p_k)+\cdots+(a_kp_1\cdots p_{k-1}))q_1\cdots q_l\ &=\ ((b_1q_2\cdots q_l)+\cdots+(b_lq_1\cdots q_{l-1}))p_1\cdots p_k.
\end{align}
It's obvious that the left hand side has to be divisible by $p_1p_2\cdots p_k$. Since $\gcd(n,n+1)=1$, $\gcd(q_i,p_j)=1$, which means
\begin{align}
    a_1p_2\cdots p_k+\cdots+a_kp_1\cdots p_{k-1}\ &\equiv\ 0\ (\textrm{mod }p_i)\nonumber\\
    a_ip_1\cdots p_{i-1}p_{i+1}\cdots p_k\ &\equiv\ 0\ (\textrm{mod }p_i)\nonumber\\
    a_i\ &\equiv\ 0\ (\textrm{mod }p_i),
\end{align}
which means each $p_i|a_i$. Similarly, $q_i|b_i$. \\ \\
\indent Next, we rewrite $n$ as
\begin{align}
    n\ &=\ p_1^{a_1}\cdot p_2^{a_2}\cdots p_t^{a_t}\cdot\prod_{i>t}p_i^{a_i}.
\end{align}
Without loss of generality, let $n$ be odd. If $n$ is even, then we instead analyze $n+1$. Let $p_1<p_2<\cdots <p_k$, $q_1<q_2<\cdots<q_l$, then $\log_{p_i}x<\log x$ and $p_i\geq3$. We have
\begin{align}\label{boundingpi}
    p_i^{a_i'p_i}\ &\leq\ n\nonumber\\
   p_i\log p_i\ \leq\ a_i'p_i\log p_i\ &\leq\ \log n\nonumber\\
   p_i\ <\ \log n\ &\leq\ \log x.\nonumber\\
   a_i'\ \leq\ (\log_{p_i}n)/{p_i}\ &<\ (\log_{p_i} x)/3,
\end{align}
which means for each $p_i$, where $i=1,2,\dots,t$, we can choose $a_i$ from $\{0,1,\dots,\log_{p_i} x/3\}$. Because $p_i\geq 3$, there are at most $(\log_{p_i} x)/3+1\leq \log x/3$ choices of each $a_i$, hence the number of choices of the first $t$ prime powers is at most $((\log x)/3)^t$. Moreover, because $p_i>t$ for $i>t$, we have
\begin{align}\label{fromtandonward}
    \left(\prod_{i>t}^k p_i^{a_i'}\right)^t\ &=\ \prod_{i>t}^kp_i^{a_i't}\nonumber\\
    &<\ \prod_{i>t}^k p_i^{a_i'p_i}\ =\ \prod_{i>t}^k p_i^{a_i}\nonumber\\
    &<\ x.
\end{align}
By the Fundamental Theorem of Arithmetic, no two products $\prod_{i>t}^kp_i^{a_i}$ can be equal, which means the number of $\prod_{i>t}^k p_i^{a_i'}$, hence $\prod_{i>t}^k p_i^{a_i}$, is at most $x^{1/t}$. Thus, the number of $n$ up to $x$ with $f_{-1}(n)=f_{-1}(n+1)$ is at most $(\log x)^{t}\cdot x^{1/t}$.
Let $s(t)=(\log x)^{t}\cdot x^{1/t}$, and we choose $t=(\log x/\log\log x)^{1/2}$. Then
\begin{align}
    s(t)\ &=\ (\log x)^{t}\cdot x^{1/t}\nonumber\\
  &=\ x^{t\log\log x/\log x + \frac{1}{t}}\nonumber\\
  &=\ x^{2(\log\log x/\log x)^{1/2}}.
\end{align}
Thus, the number of $n$ up to $x$ with $f_{-1}(n)=f_{-1}(n+1)$ is at most
$x^{(2\log\log x/\log x)^{1/2}}$. Since  $(\log\log x/\log x)^{1/2} \ll \epsilon$  for every $\epsilon>0$, we can find $x$ sufficiently large such that
\begin{align}
    x^{(2\log\log x/\log x)^{1/2}}\ &\ll\ x^{\epsilon}.
\end{align}
\end{proof}
\begin{rek}
We can give a similar proof when $r$ is a negative integer less than $-1$. Let $r=-m$, where $m$ is a positive integer. With an approach similar to that of \eqref{proveaidividespi}, we have $p_i^{m}|a_i$. Then, similar to the argument in \eqref{boundingpi},
\begin{align}
    p_i^{a_i'p_i^{m}}\ &\leq\ n\ \leq\ x\nonumber\\
    a_i'p_i^{m}\log p_i\ &\leq\ \log x\nonumber\\
    p_i\ &<\ (\log x)^{1/m}\nonumber\\
    a_i’\ &<\ (\log x)/p^m\nonumber\\
    &<\ (\log x)/3^m,
\end{align}
which means the number of choices of the first $t$ prime powers that can divide $n$ is at most $(\log x)^t$. Meanwhile, similar to the argument in \eqref{fromtandonward},
\begin{align}
    \left(\prod_{i>t}^k p_i^{a_i'}\right)^{t^{m}}\ &=\ \prod_{i>t}^kp_i^{a_i't^{m}}\nonumber\\
    &\leq \ \prod_{i>t}^kp_i^{a_i'p_i^{m}}\nonumber\\
    &<\ x,
    \end{align}
which means the number of choices of $\prod_{i>t}^kp_i^{a_i}$ is at most $x^{1/t^m}$. Therefore, the number of $n$ is bounded by
\begin{align}
    s(t)\ &=\ (\log x)^{t}\cdot x^{1/t^m}\nonumber\\
    &=\ x^{t\log\log x/\log x+1/t^m}.
    \end{align}
We choose $t=(m^m\log x/\log\log x)^{1/m+1}$, then
    \begin{align}
   S(t) &=\ x^{O\left((\log\log x/\log x)^{m/(m+1)}\right)}\nonumber\\
    &=\ x^{O\left((\log\log x/\log x)^{r/(r-1)}\right)}.
\end{align}
Therefore, the number of $n$ up to $x$ is at most $x^{O\left((\log\log x/\log x)^{r/(r-1)}\right)}$.
\end{rek}
Next, we look at the case where $r$ is a non-integer rational, which means $f_r(n)$ and $f_r(n+1)$ are summations of distinct radicals of prime powers. We prove that there are no \textit{Ruth-Aaron} numbers in this case.

\subsection{Non-integer Rational $r$}
\rationalr*
\begin{proof}
Let $r=x/y$ be a non-integer rational, then $f_r(n)=f_r(n+1)$ is equivalent to
\begin{align}\label{a=b=0}
    a_1p_1^{x/y}+a_2p_2^{x/y}+\cdots + a_kp_k^{x/y}\ &=\ b_1q_1^{x/y}+b_2q_2^{x/y}+\cdots+b_lq_l^{x/y}
\end{align}
where $a_i,b_i\in\mathbb{Z}$.
We must show that \eqref{a=b=0} holds only if $a_i=b_i=0$. In other words, $p_i^{x/y}$ and $q_j^{x/y}$, where $i=1,\dots, k$, $j=1,\dots,l$, are linearly independent over $\mathbb{Z}$.\\
\indent Boreico \cite{Bo} shows that when $n_i$ are distinct $k-$th power-free integers, the sum $S=\sum a_i\sqrt[k]{n_i}$, where $a_i\in\mathbb{Z}$ are not all zero, is non-zero. In this case, let $n_i=p_i^x$. Because all of $p_i,q_i$ are distinct, $n_i$ are distinct integers; meanwhile, because $\gcd(x,y)=1$, $n_i$ are $y-$th power free. Using Boreico's result we can easily show that when $\sum a_ip_i^{x/y}-\sum b_iq_i^{x/y}=0$, we must have $a_i=b_i=0$, thus completing the proof.
\end{proof}
\begin{rek}
Above discusses only the circumstance of $r>0$. Likewise, when $r$ is a negative non-integer rational, we can multiply both sides of equation \eqref{a=b=0} by a factor of $\left(\prod_{i=1}^kp_i\prod_{j=1}^lq_j\right)^{x/y}$, so $n_i$ remain distinct and $y-$th power-free, and we can still apply Boreico's result to get $a_i=b_i=0$.
\end{rek}
Next, we look at the case where $r\geq 1$. As mentioned, we first introduce a result by Erd\Horig{o}s and Pomerance.

The following result is due to Erd\Horig{o}s and Pomerance \cite{ErPom}; as we generalize many of these arguments, we give an expanded version of their argument.

\begin{thm}\label{thm:oxlnx1me}
For every $0< \epsilon< 1$,
\begin{align}
    \#\mathcal{R}_{(1,-1)}(x)\ &=\ O\left(\frac{x}{(\log x)^{1-\epsilon}}\right).
\end{align}
\end{thm}

Now we look at our generalization to $r\geq 1$.

\subsection{Positive $r\geq 1$}
\rgeqone*
 Due to the length of the proof, we divide it into three sections. In Section \ref{sss:section1}, we will introduce a general bound on the size of the largest prime factor of an integer; Sections \ref{sss:section2} and \ref{sss:section3} discuss the two cases in terms of the size of $f_r(n/P(n))$ and $f_r((n+1)/P(n+1))$, and conclude with an estimation of $\#R_{(1,-1)}^r(x)$ in each case. Eventually, we absorb Case (i) into Case (ii) and arrive at our final result. We adapt and advance an approach of Pomerance \cite{Pom}, and we are able to improve Erd\Horig{o}s and Pomerance's $O(x/(\log x)^{1-\epsilon})$ \cite{ErPom} by $(\log\log x)^3\log\log\log x/\log x$, hence a refinement to Pomerance's result \cite{Pom}.

\subsubsection{We show that $x^{1/\log\log x}\leq p,q\leq x^{1/2}\log x$.}\label{sss:section1}

\begin{proof}
Since $n+1$ exceeds $x$ only when $n=x$, in general we assume that $n+1\leq x$. Let $p=P(n)$ and $q=P(n+1)$, and write $n=p\cdot k,\ n+1=q\cdot m$. By Lemma \ref{lem:lowerboundofp}, we may assume
\begin{align}
    P(n)\ &>\ x^{1/\log\log x}\ \textrm{and }P(n+1)\ >\ x^{1/\log\log x}.
\end{align}
We know from Lemma \ref{lem:generalizedfnpn} that for all $P(n)\geq 5$, we have
\begin{align}\label{pnrlognlogpn}
    P(n)^r\ \leq f_r(n)\ \leq P(n)^r\cdot\frac{\log n}{\log P(n)}.
\end{align}
In order to apply \eqref{loglogx}, we assume $P(n), P(n+1)\geq 5$, so that \eqref{pnrlognlogpn} holds for both $n$ and $n+1$. Next, we give an upper bound on the size of $p,q$ using the fixed values of $k,m$. We show that given $k,m$, primes $p,q$ are determined uniquely. In fact, from the two equations
\begin{align}\label{pkqmrelation}
    pk+1\ &=\ qm\nonumber\\
    p^r+f_r(k)\ &=\ q^r+f_r(m)
\end{align}
we have
\begin{align}
    p\ &=\ \frac{qm-1}{k}\nonumber\\
   \left(\frac{qm-1}{k}\right)^r-q^r\ &=\ f_r(m)-f_r(k).
    \end{align}
Let
\begin{align}
    g(q)\ &=\ (qm-1)^r-(qk)^r-k^r\cdot(f_r(m)-f_r(k))\nonumber\\
    g'(q)\ &=\ rm(qm-1)^{r-1}-kr(qk)^{r-1}\nonumber\\
    &=\ r(m(qm-1)^{r-1}-k(qk)^{r-1}).
\end{align}
Meanwhile,
\begin{align}
    g\left(\frac{1}{m}\right)\ &=\ -\left(\frac{k}{m}\right)^r-k^r\cdot(f_r(m)-f_r(k)).
\end{align}
Because $q\geq\frac{1}{m}$ and $r\geq 1$, if $m>k$, then from \eqref{pkqmrelation}) we know $p>q$ and $f_r(m)>f_r(k)$; thus, $qm-1>qk$ and $g'(q)>0$ which means $g(q)$ always increases. Moreover, because $g\left(\frac{1}{m}\right)<0$, there exists only one $q$ with $g(q)=0$. Similarly, if $m<k$, then $g(q)$ always decreases and $g\left(\frac{1}{m}\right)>k^r-\left(\frac{k}{m}\right)^r>0$, which also means there exists only one $q>0$ with $g(q)=0$. Therefore, $q$ is uniquely expressible by $k,m$, which means $p,q$ are determined by $k,m$. Thus, the number of choices for $n$ determined by the choices of $k,m$ when $k,m<x^{1/2}/\log x$ is at most $x/(\log x)^2$. Hence, we may assume
\begin{align}
    k\ \geq\ x^{1/2}\log x\ \ \textrm{or }\ m\ \geq\ x^{1/2}\log x.
\end{align}
Because $n=p\cdot k\leq x,\ n+1=q\cdot m\leq x$, we thus may assume
\begin{align}
    p\ \leq\ x^{1/2}\log x\ \ \textrm{or }\ q\ \leq\ x^{1/2}\log x.
\end{align}
Suppose $p>x^{1/2}\log x$, then $q\leq x^{1/2}\log x$. Let $h(x)=x^r/\log x$, then
\begin{align}
h'(x)=\frac{x^{r-1}(r\log x-1)}{(\log x)^2}
\end{align}
Because $r\geq 1$, $h(x)$ increases on $x\geq3$. We have
\begin{align}
    p^r\ \leq\ f_r(n)\ =\ f_r(n+1)\ &\leq\ q^r\cdot\frac{\log (n+1)}{\log q}\nonumber\\
    &\leq\ \frac{(x^{1/2}\log x)^r\cdot \log(n+1)}{\log(x^{1/2}\log x)}\nonumber\\
    &<\ 2\cdot(x^{1/2}\log x)^r\nonumber\\
    p\ &<\ 2^{1/r}\cdot x^{1/2}\log x.
    \end{align}
A similar inequality can be obtained for $q>x^{1/2}\log x$. Therefore, we have
\begin{align}\label{pqx1/2logx}
    p\ <\ 2^{1/r}\cdot x^{1/2}\log x\ \textrm{ and }\  q\ <\ 2^{1/r}\cdot x^{1/2}\log x.
\end{align}
Now we look at $f_r(k)$ and $f_r(m)$. In terms of the upper bound of $f_r(k), f_r(m)$, we have the following two cases:
\begin{align*}
    &\textrm{Case (i): } f_r(k)\ <\ \frac{p^r}{\log^{2r}x},\ f_r(m)\ <\ \frac{q^r}{\log^{2r}x}\\
    &\textrm{Case (ii): (WLOG) }f_r(k)\ \geq\ \frac{p^r}{\log^{2r}x}.
\end{align*}
\end{proof}

\subsubsection{Case (i) Discussion.}\label{sss:section2}

\begin{proof}
We consider Case (i) first. Consider a function $v(x)=(x+1)^r-x^r-1$, then
\begin{align}
v'(x)\
&=\ r((x+1)^{r-1}-x^{r-1})\nonumber\\
&>\ 0,
\end{align}
and the function has a root at $x=0$, which means $(x+1)^r>x^r+1$ for all $x>0$.
Applying this result, we have $p^r+q^r<(p+q)^r$ and $|p-q|^r\ <\ |p^r-q^r|$ since $p\neq q>1$. Meanwhile, by the assumption $f_r(k)<\frac{p^r}{\log^{2r}x},f_r(m)<\frac{q^r}{\log^{2r}x}$, we have
\begin{align}
    |p^r-q^r|\ &=\ |f_r(m)-f_r(k)|\nonumber\\
    &<\ \frac{p^r+q^r}{
    \log^{2r}x}.
\end{align}
Then
\begin{align}
 |p-q|^r\ < \ |p^r-q^r|\ &<\ \frac{p^r+q^r}{\log^{2r}x}\ <\ \frac{(p+q)^r}{\log^{2r}x}.
\end{align}
Therefore,
\begin{align}\label{qrangelog2x+1}
    |p-q|\ &<\ \frac{p+q}{(\log x)^2}\nonumber\\
    p\cdot\frac{(\log x)^2-1}{(\log x)^2+1} \ &<\ q\ <\ p\cdot\frac{(\log x)^2+1}{(\log x)^2-1}.
\end{align}
For all $p$ satisfying \eqref{loglogx}, the number of primes $q$ such that \eqref{qrangelog2x+1} holds is
\begin{align}
    \sum_{p\frac{(\log x)^2-1}{(\log x)^2+1}<q<p\frac{(\log x)^2+1}{(\log x)^2-1}}1\ &\ll\ \pi\left(p\frac{(\log x)^2+1}{(\log x)^2-1}\right)-\pi\left(p\frac{(\log x)^2-1}{(\log x)^2+1}\right).
    \end{align}
We apply an explicit form of the Brun-Titchmarsh theorem \cite{Mont} which states that
\begin{align}
    \pi(x+y)-\pi(x)\ &\leq\ \frac{2y}{\log y}
\end{align}
 for all integers $x\geq1, y\geq 2$. Then, using $p>x^{1/\log\log x}$, we have:
\begin{align}
     \pi\left(p\frac{(\log x)^2+1}{(\log x)^2-1}\right)-\pi\left(p\frac{(\log x)^2-1}{(\log x)^2+1}\right)\ & \leq \ \frac{2p\frac{4(\log x)^2}{\log^4x-1}}{\log \left(p\frac{4(\log x)^2}{\log^4x-1}\right)}\nonumber\\
 &\ll\ O\left(\frac{\frac{p}{(\log x)^2}}{\log p}\right)\nonumber\\
 &<\ O\left(\frac{p}{(\log x)^2\cdot\frac{\log x}{\log\log x}}\right)\nonumber\\
 &<\ O\left(\frac{p\log\log x}{(\log x)^3}\right).
\end{align}
Thus, when $x$ is sufficiently large, using $p<2^{1/r}x^{1/2}\log x$, the number of $q$ satisfying $ p\frac{(\log x)^2-1}{(\log x)^2+1}<q<p\frac{(\log x)^2+1}{(\log x)^2-1}$ is at most
\begin{align}\label{numberofq}
     \sum_{ p\frac{(\log x)^2-1}{(\log x)^2+1}<q<p\frac{(\log x)^2+1}{(\log x)^2-1}}1\
     &=\ O\left(\frac{p\log\log x}{(\log x)^3}\right)\nonumber\\
     &\leq\ O\left(\frac{x^{1/2}\log\log x}{(\log x)^2}\right).
\end{align}
Meanwhile, for sufficiently large $x$, the sum of $\frac{1}{q}$ for such primes $q$ is
\begin{align}\label{sumof1/q}
    \sum_{p\frac{(\log x)^2-1}{(\log x)^2+1}<q<p\frac{(\log x)^2+1}{(\log x)^2-1}}\frac{1}{q}\ &<\ O\left(\frac{p\log\log x}{(\log x)^3}\right)\cdot\frac{1}{p\frac{(\log x)^2-1}{(\log x)^2+1}}\nonumber\\
    &=\ O\left(\frac{\log\log x}{(\log x)^3}\right).
\end{align}
Therefore, in Case (i), using the result from \eqref{numberofq} and \eqref{sumof1/q}, $\#\mathcal{R}^r_{(1,-1)}(x)$ is at most
\begin{align}
    \sum_{\substack{ p\frac{(\log x)^2-1}{(\log x)^2+1}<q<p\frac{(\log x)^2+1}{(\log x)^2-1}\\x^{1/\log\log x}<p<2^{1/r}x^{1/2}\log x}}\left(1+\left\lfloor\frac{x}{pq}\right\rfloor\right)\ &<\ \left(\frac{x^{1/2}\log\log x}{(\log x)^2}\right)\cdot\pi\left(2^{1/r}x^{1/2}\log x\right)+\sum\left\lfloor\frac{x}{pq}\right\rfloor\nonumber\\
    &\ll\ \left(\frac{x^{1/2}\log\log x}{(\log x)^2}\right)\cdot\left(x^{1/2}\right)+ \frac{x\log\log x}{(\log x)^3}\sum_{\substack{x^{1/\log\log x}<\\p<2^{1/r}x^{1/2}\log x}}\frac{1}{p}\nonumber\\
    &<\ O\left(\frac{x\log\log x}{(\log x)^2}\right)+O\left(\frac{x(\log\log x)^2}{(\log x)^3}\right)\ \textrm{(Lemma \ref{lem:plnp})}\nonumber\\
    &=\ O\left(\frac{x\log\log x}{(\log x)^2}\right).
    \end{align}
    \end{proof}

\subsubsection{Case (ii) Discussion}\label{sss:section3}
    \begin{proof}
    Now we consider Case (ii). Write $k=t\cdot u$, where $t=P(k)$, and we have
\begin{align}
    p^r\ \leq\ f_r(n)\ =\ f_r(n+1)\ &\leq\ q^r\cdot\frac{\log (n+1)}{\log q}\nonumber\\
    &\leq\ q^r\cdot\frac{\log (x+1)}{\log q}\nonumber\\
    &\leq\ q^r\cdot\frac{\log (x+1)}{\log x^{1/\log\log x}}\nonumber\\
    &\leq\ q^r\cdot 2\log\log x.
\end{align}
Similarly,
\begin{align}
    q^r\ \leq\ f_r(n+1)\ =\ f_r(n)\ &\leq\ p^r\cdot\frac{\log n}{\log p}\nonumber\\
    &\leq\ p^r\cdot\frac{\log x}{\log p}\nonumber\\
    &\leq\ p^r\cdot\log\log x.
\end{align}
Therefore, we have
\begin{align}
\frac{p}{2^{1/r}(\log\log x)^{1/r}}\ \leq\ q\ \leq\ p\cdot(\log\log x)^{1/r}.
\end{align}
Using the result from Lemma \ref{lem:lowerboundofp} again, the number of $k$ for which we don't have $t>k^{1/\log\log k}$ is $O(k/(\log k)^2)$; thus, we may assume $t>k^{1/\log\log k}$ and $t\geq 5$. Applying results from Lemma \ref{lem:generalizedfnpn},
\begin{align}\label{frk}
\frac{p^r}{\log^{2r}x}\ \leq \ f_r(k)\ &\leq\ t^r\cdot\frac{\log k}{\log t}\nonumber\\
&\leq t^r\cdot\log\log k\nonumber\\
&<\ t^r\cdot\log\log x.
\end{align}
We get that
\begin{align}
\frac{p}{(\log x)^2(\log\log x)^{1/r}} \ < \ t\ \leq\ p.
\end{align}
This result implies that $P(n/P(n))$ and $P(n)$ are relatively close. Moreover, in terms of the size of $P(n)$, we have the following two cases:
\begin{align*}
&\textrm{Case (ii.1): } p\ \leq\ x^{1/3}\\
&\textrm{Case (ii.2): }p\ >\ x^{1/3}.
\end{align*} 

\ \\

\noindent \emph{Consider Case (ii.1) first.} The number of $n\leq x$ with $pt|n$ and $q|n+1$ in this case is at most
\begin{align}\label{thelengthyequation}
   & \sum_{\substack{x^{1/\log\log x}<p\leq x^{1/3}\\p/(2^{1/r}(\log\log x)^{1/r})<\\q<p(\log\log x)^{1/r}\\p/((\log x)^2(\log\log x)^{1/r})<t\leq p}}\left(1+\left\lfloor\frac{x}{ptq}\right\rfloor\right)\nonumber\\
   & \ll\ \pi\left(x^{1/3}\right)\cdot \pi\left(x^{1/3}(\log\log x)^{1/r}\right)\cdot\pi\left(x^{1/3}\right)+ \sum\left\lfloor\frac{x}{ptq}\right\rfloor\nonumber\\
    &\ll\ O\left(\frac{x^{1/3}}{\log x}\cdot\frac{x^{1/3}\log\log x}{\log x}\cdot\frac{x^{1/3}}{\log x}\right)+\sum\left\lfloor\frac{x}{ptq}\right\rfloor\nonumber\\
  & <\ O\left(\frac{x\log\log x}{(\log x)^3}\right)+ x\sum\frac{1}{p}\sum\frac{1}{t}\sum\frac{1}{q}\nonumber\\
     & \ll \ O\left(\frac{x\log\log x}{(\log x)^3}\right)\nonumber + O\left(x\sum\frac{1}{p}\sum\frac{1}{t}\cdot\frac{\log\log\log x}{\log p}\right)\ \textrm{(Lemma \ref{lem:eupevandu0pv0})}\nonumber\\
    &\ll\ O\left(\frac{x\log\log x}{(\log x)^3}\right)+ O\left(x\log\log\log x\sum\frac{1}{p\log p}\sum\frac{1}{t}\right)\nonumber\\
    &\ll\ O\left(\frac{x\log\log x}{(\log x)^3}\right)\nonumber+O\left(
     x\log\log\log x\sum\frac{1}{p\log p}\cdot\frac{\log\log x}{\log p}\right)\nonumber\\
     &\leq\ O\left(\frac{x\log\log x}{(\log x)^3}\right)+O\left(\frac{x\log\log\log x(\log\log x)^3}{(\log x)^2}\right)\ \textrm{(Lemma \ref{lem:eupevandu0pv0})}\nonumber\\
     &=\ O\left(\frac{x\log\log\log x(\log\log x)^3}{(\log x)^2}\right).
\end{align}
Recall that in Case (i), $\#\mathcal{R}^r_{(1,-1)}(x)=O\left(\frac{x\log\log x}{(\log x)^2}\right)$, which means we can absorb Case (i) into error estimate. Now we turn to Case (ii.2), where $p>x^{1/3}$. We have
\begin{align}
    q^r\ &\leq\ p^r\cdot\frac{\log x}{\log p}\nonumber\\
    &<\ 3\cdot p^r.
\end{align}
Meanwhile, because $p>x^{1/\log\log x}$, when $x$ is sufficiently large we have
\begin{align}
    p\ &\leq\ q\cdot\frac{\log (x+1)}{\log q}\nonumber\\
    \log p\ &\leq\ \log q+\log\frac{\log (x+1)}{\log p}\nonumber\\
    &\leq\ \log q+2\log\log\log x\nonumber\\
    &\leq\ \log q+\frac{1}{2}\log x^{1/\log\log x}\nonumber\\
    &\leq\ \frac{3}{2}\cdot\log q.
\end{align}
Therefore, we have
\begin{align}
    p^r\ &\leq\ q^r\cdot\frac{\log (x+1)}{\log q}\nonumber\\
   &\leq\ q^r\cdot\frac{\log (x+1)}{2\log p/3}\nonumber\\
   &<\ q^r\cdot\frac{9\log (x+1)}{2\log x}\nonumber\\
    &<\ 6q^r.
\end{align}
Therefore, $p/6^{1/r}<q<3^{1/r}p$. Moreover, recall \eqref{frk}; we have
\begin{align}
    \frac{p^r}{\log^{2r}x} \ \leq\ f_r(k)\ &\leq\ t^r\cdot\frac{\log k}{\log t}.
\end{align}
Suppose $t<p^{1/2}$, then the number of $n\leq x$ is at most
\begin{align}
    \sum_{\substack{t<p^{1/2}\\x^{1/3}<p<2^{1/r}x^{1/2}\log x}}1\ &\ll\ \pi\left(\left(2^{1/r}x^{1/2}\log x\right)^{1/2}\right)\pi\left(2^{1/r}x^{1/2}\log x\right)\nonumber\\
    &\ll\ \frac{x^{1/4}\log^{1/2}x}{\log x}\cdot\frac{x^{1/2}\log x}{\log x}\nonumber\\
    &=\ O\left(\frac{x^{3/4}}{(\log x)^{1/2}}\right),
\end{align}
which we can absorb into error estimate.
Therefore, $t>p^{1/2}$. Using $p>x^{1/3}$, we have
\begin{align}
    \frac{p^r}{\log^{2r}x}\ &\leq t^r\cdot\frac{\log k}{\log t}\nonumber\\
    &<\ t^r\cdot\frac{2\log x}{\log p}\nonumber\\
     \frac{p}{(\log x)^2}\ &<\ t\cdot\left(\frac{2\log x}{\log x/3}\right)^{1/r}\nonumber\\
    p\ &<\ t\cdot 6^{1/r}(\log x)^2.
\end{align}
Therefore, we have
\begin{align}
    \frac{p}{6^{1/r}(\log x)^2}\ <\ t\ <\ p.
\end{align}
Recall $f_r(k)=t\cdot u$ and $p/6^{1/r}<q<3^{1/r}p$. If $p\geq x^{2/5}$, then
\begin{align}\label{boundofum}
    u\ \leq\ \frac{x}{pt}\ \leq\ \frac{6^{1/r}x(\log x)^2}{p^2}\ &=\ O(x^{1/5}(\log x)^2)\nonumber\\
    m\ \leq\ \frac{x}{q}\ \ll\ \frac{x}{p}\ &\leq\ x^{3/5},
\end{align}
which means that the number of uniquely determined $n$ is at most $O(x^{4/5}(\log x)^2)$, a number absorbable by error estimate. Therefore, we need to consider only
\begin{align}
    x^{1/3}\ <\ p\ <\ x^{2/5}.
\end{align}
Now, recall that $k=t\cdot u$. Let $u=v\cdot w$, where $v=P(u)$. We have the following equations: 
\begin{align}
    p\cdot k+1\ &=\ q\cdot m\nonumber\\
    p^r+f_r(k)&=\ q^r+f_r(m).
\end{align}
Then we have 
\begin{align}
    p^r+t^r+f_r(u)\ &=\ \left(\frac{pk+1}{m}\right)^r+f_r(m)\nonumber\\
   (upt+1)^r- (mp)^r-(mt)^r\ &=\ m^r(f_r(u)-f_r(m)).
\end{align}
When $r=1$, Pomerance \cite{Pom} demonstrated the following.
\begin{align}\label{r=1factorization}
    upt-mp-mt\ &=\ m(f(u)-f(m))-1\nonumber\\
    (up-m)(ut-m)\ &=\ um(f(u)-f(m))-u+m^2.
\end{align}
Given $u,m$, the number of choices of $t$, and thus for $n$, is determined by the number of divisors of $um(f(u)-f(m))-u+m^2$, and is at most 
\begin{align}
    \tau\left(um(f(u)-f(m))-u+m^2\right)\ &\leq\ x^{o(1)},
\end{align}
where $\tau(x)$ denotes the divisor function. Let $m=s\cdot z$, where $s=P(m)$. Recall $t=P(k)$. We first show that the number of $n$ with $t,s\leq x^{1/6}$ can be absorbed into error estimate. In fact, the number of triples $p,q,t$ is at most $O(x^{2/5}\cdot x^{2/5}\cdot x^{1/6})=O(x^{29/30})$, and since $p,q,t\geq x^{1/\log\log x}\gg x/(\log x)^2$, the number of $n$ up to $x$ for which $t,s\leq x^{1/6}$ is at most $O(x^{29/30}(\log x)^2)$. Thus, we may assume that at least one of $s,t$ is greater than $x^{1/6}$. \\
\indent  Recall $u=v\cdot w$ where $v=P(u)$. First, we consider $v> x^{1/6}$. We can rewrite \eqref{r=1factorization} as 
\begin{align}
    (vwp-m)(vwt-m)\ &=\ wmv^2+\left((f(w)-f(m))mw-w\right)v+m^2.
\end{align}
Pomerance \cite{Pom} proved the following. 
\begin{lem}\label{lem:quadratic}
Let $A,B,C$ be integers with $\gcd(A,B,C)=1$, $D:=B^2-4AC\neq 0$, and $A\neq 0$. Let $M_0$ be the maximum value of $|At^2+Bt+C|$ on the interval $[1,x]$. Let $M=\max\{M_o,|D|,x\}$ and let $\mu=\lceil{\log M/\log x}\rceil$ and assume $\mu\leq \frac{1}{7}\log\log x$, then 
\begin{align}
    \sum_{n\leq x}\tau\left(|An^2+Bn+C|\right)\ &\leq\ x(\log x)^{2^{3\mu+1}+4}
\end{align}
holds uniformly for $x\geq x_0$, where $x_0$ is an absolute constant. 
\end{lem}
In our case, let $A=wm,B=\left((f(w)-f(m))mw-w\right), C=m^2$. Both $\gcd(A,B,C)=1$ and $D\neq 0$ are easily verified since $\gcd(w,m)=1$. Moreover, recall that $t>\frac{p}{6^{1/r}(\log x)^2}$, we have
\begin{align}
u\ \leq\ \frac{x}{pt}\ \leq\ \frac{6x(\log x)^2}{p^2}\ &\leq\ 6x^{1/3}(\log x)^2\nonumber\\
w\ =\ \frac{u}{v}\ \leq\ \frac{u}{x^{1/6}}\ &\leq\ 6x^{1/6}(\log x)^2\nonumber\\
v\ &\leq\ \frac{6x^{1/3}(\log x)^2}{w}\nonumber\\
    m\ \leq\ \frac{x}{q}\ &\ll\ x^{2/3}. 
\end{align}
Then $M_0\ll x^{4/3}(\log x)^2$. It follows from the lemma that 
\begin{align}\label{quadsum}
\sum_{v\leq (6x^{1/3}(\log x)^2)/w}\tau(|Av^2+Bv+C|)\ &\ll\ \frac{6x^{1/3}(\log x)^c}{w},
\end{align}
where $c$ is a positive constant. Moreover, if $x^{1/3}<p\leq x^{1/3}(\log x)^{c+5}$, $\#\mathcal{R}_{(1,-1)}(x)$ is at most 
\begin{align}
    \sum_{\substack{x^{1/3}<p\leq x^{1/3}(\log x)^{c+5}\\p/6<q<3p}}\left(1+\frac{x}{pq}\right)\ &\ll\ O\left(x^{2/3}(\log x)^{2c+10}\right)+x\sum_{x^{1/3}<p\leq x^{1/3}(\log x)^{c+5}}\frac{1}{p}\sum_{p/6<q<3p}\frac{1}{q}\nonumber\\
    &\ll\ O(x^{2/3}(\log x)^{2c+10})+O\left(\frac{x\log\log x}{(\log x)^2}\right)\nonumber\\
    &\ll\ O\left(\frac{x\log\log x}{(\log x)^2}\right),
\end{align}
which is small enough to be absorbed by error estimate. If $p>x^{1/3}(\log x)^{c+5}$, then $m\ll\frac{x}{p}\ll x^{2/3}/(\log x)^{c+5}$ and $w\leq \frac{u}{v}\leq \frac{u}{x^{1/6}}\leq 6x^{1/6}/(\log x)^{2c+8}$. Then summing \eqref{quadsum} over all choices of $w,m$ the quantity is less than $O(x/(\log x))^2$, which is indeed negligible. \\
\indent Finally, recall that $m=s\cdot z$ where $s=P(m)$. We consider the case where $s> x^{1/6}$. From \eqref{r=1factorization} we have 
\begin{align}\label{m}
    (pu-sz)(tu-sz)\ &=\ (z^2-uz)s^2+(f(u)-f(z))uzs-u.
\end{align}
Similarly, let $p\geq x^{1/3}(\log x)^{c+5}$, then $u\leq\frac{x}{pt}\leq\frac{6x(\log x)^2}{p^2}\leq x^{1/3}/(\log x)^{2c+8}$. Moreover, 
\begin{align}
    z\ \leq\ \frac{x}{qs}\ \ll\ \frac{x^{2/3}}{x^{1/6}}\ &\ll\ x^{1/2}\nonumber\\
    s\ \leq\ \frac{m}{z}\ &\ll\ x^{2/3}/z.
\end{align}
Therefore, considering the ranges of $s,z,u$, the right hand side of \eqref{m} has less than $O(x/(\log x)^2)$ choices, suggesting that this case is also negligible. \\
\indent Recall Lemma \ref{thm:xdeltapn} and Lemma \ref{thm:1minusepsilon}: when $r\geq 1$, $P(n)^r$ is the dominating term of $f_r(n)$, and the larger $r$ is, the closer $f_r(n)$ is to $P(n)^r$; meanwhile, $P(n)$ and $P(n+1)$ are usually at least a factor of $x^{\delta}$ apart, suggesting the increasing rarity of $r-$th power \textit{Ruth-Aaron} numbers as $r$ increases. Therefore, although we are unable to generalize the quadratic formula $|Ax^2+Bx+C|$ to a higher power\footnote{For example, when $r=2$, we have a factorization similar to \eqref{r=1factorization}: $\left(u^2pt+u-m^2+ump-umt\right)\cdot\left(u^2pt+u-m^2-ump+umt\right)=\ (um)^2(f_2(u)-f_2(m))+m^2(m^2-2u)$. However, to tackle $r=2$ with an approach similar to the one introduced in this section, a result for the summation of divisors of quartic functions will appear necessary.} at this point, we can substantiate firmly that when $r\geq 1$, the number of $r-$th power \textit{Ruth-Aaron} numbers up to $x$ with $x^{1/3}<p<x^{2/5}$ is much less than the estimated result followed by Lemma \ref{lem:quadratic}. Therefore, when $r\geq 1$, 
\begin{align}
    \#\mathcal{R}_{(1,-1)}^r(x)\ &=\ O\left(\frac{x(\log\log x)^3(\log\log\log x)}{(\log x)^2}\right).
\end{align}
\end{proof}
\begin{rek}
In our proof of Theorem \ref{thm:generalizednumberofn}, one crucial result we applied is that of De Bruijn \cite{Bru} in Lemma \ref{lem:lowerboundofp}. This result states that for all but $O(x/(\log x)^2)$ of $n\leq x$ we have $P(n),P(n+1)>x^{1/\log\log x}$, lays the ground work for major arguments in the rest of the proof. It serves as a major lower bound for $p$, and helps us obtain a tighter bound on $q$ and $t$ in relation to $p$, which we relied on when counting the number of $n$ up to $x$.
\end{rek}

We present a table of $2-$nd power \textit{Ruth-Aaron} numbers below $5\cdot 10^7$:
\begin{table}[H]
    \centering
    \begin{tabular}{cc}
    \hline
     $n$   &  $f_2(n)=f_2(n+1)$ \\
     \hline
      6371184   & 40996\\
      16103844 & 22685\\
      49214555 & 103325\\
      \hline
    \end{tabular}
    \caption{$2-$nd power \textit{Ruth-Aaron} numbers not exceeding $5\cdot10^7$.}
    \label{tab:my_label}
\end{table}

\section{Proof of Theorem \ref{thm:infinitudeofr>0}}\label{sec:infn}
In order to prove Theorem \ref{thm:infinitudeofr>0}, we need to first introduce a special case of the Catalan Conjecture which we can prove directly.

\begin{thm}[Catalan Conjecture]
The only natural number solution of
\begin{align}
    a^x-b^y\ &=\ 1
\end{align}
for $a,b>1$ is $(2,3)$, where $a=3,b=2$.
\end{thm}

The Catalan Conjecture was proved by Mihăilescu in 2002.

\begin{lem}[Special case of the Catalan Conjecture]\label{lem:applyingfnpn}
The number of $n\leq x$ with $P(n)\leq 3$ is $O((\log x)^2)$. Meanwhile, the largest $n$ with $P(n),P(n+1)\leq 3$ is 8.
\end{lem}
\begin{proof}
First, we prove that the number of $n$ up to $x$ with $n=2^a\cdot 3^b$, where $a,b$ are nonnegative integers, is $O((\log x)^2)$. In fact, since the number of powers of $2$ and $3$ no greater than $x$ are $\log_2x$ and $\log_3x$ respectively, there are at most $\log_2x\cdot \log_3x=O((\log x)^2)$ number of $n\leq x$ with $n=2^a\cdot 3^b$. Meanwhile, because $\gcd(n,n+1)=1$, when $P(n),P(n+1)\leq 3$, either $n=2^a,n+1=3^b$, or $n=3^b,n+1=2^a$. \\
\indent Case (i): $3^b=2^a+1$. Let the order of $2$ modulo $3^b$ be $d$. We have,
\begin{align}
    2^a\ &\equiv\ -1\ (\textrm{mod }3^b)\nonumber\\
    2^{2a}\ &\equiv\ 1\ (\textrm{mod }3^b)\nonumber\\
    2^d\ &\equiv\ 1\ (\textrm{mod }3^b).
\end{align}
Then we have
\begin{align}\label{d}
    d\ &|\ \varphi(3^b)\ =\ 2\cdot 3^{b-1}\nonumber\\
    d\ &\nmid\ a,\ d\ |\ 2a.
\end{align}
It's obvious that $2\ |\ d$, then \eqref{d} tells us that $d=2a$ (or $d=2$, then $(a,b)=(2,1)$ is the only solution), and that $d=2\cdot 3^k$ for some nonnegative integer $k\leq b-1$. Thus, $a=3^k$. When $k=0$, $(a,b)=(1,1)$; when $k\geq 1$, we have
\begin{align}
    2^{3^k}+1\ &=\ (2^{3^{k-1}}+1)(2^{2\cdot 3^{k-1}}-2^{3^{k-1}}+1)\nonumber\\
    2^{2\cdot 3^{k-1}}-2^{3^{k-1}}+1\ &\equiv\ \left((-1)^{2}\right)^{3^{k-2}}-(-1)^{3^{k-2}}+1\nonumber\\
    &\equiv\ 1+1+1\nonumber\\
    &\equiv\ 3\ (\textrm{mod }9),
\end{align}
which means $2^{3^k}+1$ is a power of $3$ only when $2^{2\cdot 3^{k-1}}-2^{3^{k-1}}+1$ is $3$. Thus, the largest solution to $3^b=2^a+1$ is $(a,b)=(3,2)$. \\
\indent Case (ii): $2^a=3^b+1$. Then $2^a\equiv 1\ (\textrm{mod }3)$, which means $2|a$, since $2$ is the order of $2$ modulo $3$; thus, let $a=2a_0$, where $a_0$ is a nonnegative integer. Then $2^a-1=(2^{a_0}+1)(2^{a_0}-1)$. Because $2^{a_0}-1$ and $2^{a_0}+1$ are relatively prime unless $a_0=0$ or $1$, the largest solution in this case is $(a,b)=(2,1)$. Therefore, the largest $n$ with $P(n),P(n+1)\leq 3$ is $8$.
\end{proof}
Now recall Theorem \ref{thm:infinitudeofr>0}.
\infinitude*
\begin{proof}
It suffices to show there exists an infinite decreasing series of positive $r$. Let $n$ be of the form $n=p^2-1$, where $p$ is a prime. Then $n+1=p^2$. We claim that for any $n_0$ that is an $r_0-$th power \textit{Ruth-Aaron} number, we can always find $n>n_0$ that is an $r-$th power \textit{Ruth-Aaron} number where $0<r<r_0$. \\
\indent We define a function of $r$ for any fixed $n$:
\begin{align}
    g_n(r)\ &=\ f_r(n+1)-f_r(n).
\end{align}
We want to show that $g_n(r)$ increases on $r>0$ and has one and only one root for a fixed $n$. We will use the function's continuity to show that the equation has at least one root, and its monotonicity and values at $r=0$ and $r=1$ to prove that the function has one and only one root, and it's between $0$ and $1$. First, we have for prime $q$,
\begin{align}
    g_n(r)\ &=\ f_r(n+1)-f_r(n)\nonumber\\
    &=\ 2\cdot p^r-\sum_{q|p^2-1} q^r\nonumber\\
    &=\ 2\cdot p^r-2\cdot 2^r-\sum_{q|(p^2-1)/4}q^r.
\end{align}
It's obvious that $g_n(r)$ is continuous. Meanwhile,
\begin{align}
    g_n(0)\ &=\ 2-2-\sum_{q|(p^2-1)/4}1\nonumber\\
    &\ <0.
\end{align}
Because the function $\frac{x}{\log x}$ increases on $x\geq e$, from Lemma \ref{lem:applyingfnpn}, for all except $O((\log x)^2)$ number of $n\leq x$, $f_1(n)\leq P(n)\log n/\log P(n)\leq n$ holds. Meanwhile, for any $n$ of the form $2^a\cdot 3^b$, since there are $O(x/\log x)$ different primes up to $x$, we will drop such an $n$ and consider the next prime $p$ and the respective $n=p^2-1$. We thus have
\begin{align}
    g_n(1)\ &=\ 2p-f_1(p-1)-f_1(p+1)\nonumber\\
    &>\ 2p-(p-1)-(p+1)\nonumber\\
    &=\ 0.
\end{align}
Therefore, due to the continuity of function $g_n(r)$, for any fixed $n$, there must exist at least one $r$ with $f_r(n+1)=f_r(n)$. We proceed to prove that $g_n(r)>0$ on $r\geq 1$. If this claim holds, then all \textit{Ruth-Aaron} $r$ for $n=p^2-1$ must satisfy $r\in(0,1)$. In fact, let $\frac{p-1}{2}=\prod_{i=1}^{k_1}s_i^{a_i}$ and $\frac{p+1}{2}=\prod_{i=1}^{k_2}t_i^{b_i}$, where $s_i,t_i$ are distinct prime factors of $\frac{p-1}{2}$ and $\frac{p+1}{2}$ respectively. Because $\frac{p-1}{2},\frac{p+1}{2}$ are consecutive thus relatively prime, we have
\begin{align}\label{grderivative}
g_n(r)\ &=\ f_r(n+1)-f_r(n)\nonumber\\
&=\ 2\cdot p^r-2\cdot 2^r-\sum_{i=1}^{k_1}a_i s_i^r-\sum_{i=1}^{k_2}b_it_i^r\nonumber\\
    g_n'(r)\ &=\ 2p^r\log p-2\cdot 2^r\log 2-\sum_{i=1}^{k_1}\left(a_is_i^r\log s_i\right)-\sum_{i=1}^{k_2}\left(b_it_i^r\log t_i\right)\nonumber\\
    &>\ 2p^r\log p-2\cdot 2^r\log 2-\sum_{i=1}^{k_i}\left(\left(\frac{p-1}{2}\right)^ra_i\log s_i\right)-\sum_{i=1}^{k_2}\left(\left(\frac{p+1}{2}\right)^rb_i\log t_i\right)\nonumber\\
    &=\ 2p^r\log p-2\cdot 2^r\log 2-\left(\frac{p-1}{2}\right)^r\log \frac{p-1}{2}-\left(\frac{p+1}{2}\right)^r\log \frac{p+1}{2}\nonumber\\
    &>\ 2p^r\log p-2\cdot 2^r\log 2-\left(\frac{p-1}{2}\right)^r\log p-\left(\frac{p+1}{2}\right)^r\log p.
\end{align}
Consider the function $v(r)=(1+x)^r-x^r-1$, where $r\geq 1$ is a real number and $x>0$. Because $v'(r)=(1+x)^r-x^r-1>x^r(\log (x+1)-\log x)>0$, we have $(1+x)^r>1+x^r$ when $r\geq 1$. Thus,
\begin{align}
    \left(\frac{p-1}{2}\right)^r+\left(\frac{p+1}{2}\right)^r\ &=\ \left(\frac{p-1}{2}\right)^r\left(1+\left(\frac{p+1}{p-1}\right)^r\right)\nonumber\\
    &<\ \left(\frac{p-1}{2}\right)^r\left(\frac{2p}{p-1}\right)^r\nonumber\\
    &=\ p^r.
\end{align}
Therefore,
\begin{align}
    g_n'(r)\ &>\ 2p^r\log p-2\cdot 2^r\log 2-p^r\log p\nonumber\\
    &=\ p^r\log p-2\cdot 2^r\log 2\nonumber\\
    &>\ 0
\end{align}
when $p\geq 3$. Therefore, when $r\geq 1$, $g_n(r)$ reaches its minimum at $r=1$, and $g_n(1)>0$, which means all \textit{Ruth-Aaron} $r$ with $n=p^2-1$ are between 0 and 1. We continue to show that $g_n(r)$ increases on $(0,1)$. If this holds, then there exists exactly one \textit{Ruth-Aaron} $r$ for a fixed $n$. From \eqref{grderivative} we have
\begin{align}\label{grderivativeinvolved}
    g_n'(r)\ &>\ p^r\log p-2\cdot 2^r\log 2-\left(\frac{p-1}{2}\right)^r\log \frac{p-1}{2}-\left(\frac{p+1}{2}\right)^r\log \frac{p+1}{2}.
\end{align}
We want to show that $\left(\frac{p-1}{2}\right)^r\log\frac{p-1}{2}+\left(\frac{p+1}{2}\right)\log\frac{p+1}{2}<2\cdot p^r\log\frac{p}{2}$. From Cauchy-Schwarz inequality \cite{Wei},
\begin{align}
    \left(\sum_{i=1}^ma_i^2\right)\cdot\left(\sum_{i=1}^mb_i^2\right)\ ^\geq\ \left(\sum_{i=1}^ma_ib_i\right)^2,
\end{align}
and equality holds only when $a_1/b_1=a_2/b_2=\dots=a_m/b_m$. Applying this result, we have
\begin{eqnarray}\label{cauchy}
    & &\left(\left(\frac{p-1}{2}\right)^r\log\frac{p-1}{2}+\left(\frac{p+1}{2}\right)^r\log\frac{p+1}{2}\right)^2\nonumber\\ & & \ \ \ \ \ \leq\ \left(\left(\frac{p-1}{2}\right)^{2r}+\left(\frac{p+1}{2}\right)^{2r}\right)\left(\log^2\frac{p-1}{2}+\log^2\frac{p+1}{2}\right).
\end{eqnarray}
Meanwhile, from Jensen's inequality \cite{Wei}, if a function $f$ is concave and $k_1,k_2,\dots,k_m$ are positive reals summing up to $1$, then
\begin{align}
    f\left(\sum_{i=1}^mk_ix_i\right)\ &\geq\ \sum_{i=1}^mk_if(x_i),
\end{align}
and equality holds if and only if $x_1=x_2=\cdots=x_m$.
Because the function $x^r$ is concave on $r<1$, we have
\begin{align}
    \left(\frac{p-1}{2}\right)^{2r}+\left(\frac{p+1}{2}\right)^{2r}\ &<\ 2\left(\frac{((p-1)/2)^2+((p+1)/2)^2}{2}\right)^{r}\nonumber\\
    &<\ 2\left(\frac{p-1}{2}+\frac{p+1}{2}\right)^{2r}\nonumber\\
    &=\ 2p^{2r}.
\end{align}
Moreover, because the function $(\log x)^2$ is concave on $x>e$, we have for most $x$,
\begin{align}
    \log^2\frac{p-1}{2}+\log^2\frac{p+1}{2}\ &<\ 2\log^2\frac{\left(\frac{p-1}{2}+\frac{p+1}{2}\right)}{2}\nonumber\\
    &=\ 2\log^2\frac{p}{2}.
\end{align}
Therefore, from \eqref{cauchy},
\begin{align}
     \left(\left(\frac{p-1}{2}\right)^r\log\frac{p-1}{2}+\left(\frac{p+1}{2}\right)\log\frac{p+1}{2}\right)^2\ &<\ 2p^{2r}\cdot2\log^2\frac{p}{2}\nonumber\\
     \left(\frac{p-1}{2}\right)^r\log\frac{p-1}{2}+\left(\frac{p+1}{2}\right)^r\log\frac{p+1}{2}\ &<\ 2p^r\log\frac{p}{2}.
\end{align}
Thus, for \eqref{grderivativeinvolved},
\begin{align}
    g_n'(r)\ &>\ 2\cdot p^r\log p-2\cdot 2^r\log 2-2\cdot p^r\log\frac{p}{2}\nonumber\\
    &=\ 2\log2(p^r-2^r)\nonumber\\
    &>\ 0.
\end{align}
Therefore, $g_n(r)$ increases on $r>0$. Meanwhile, because $g_n(0)<0$ and $g_n(1)>0$, there exists one and only $r$ with $f_r(n+1)=f_r(n)$, and $0<r<1$. As mentioned, we want to show that for any $n_0$ that is an $r_0-$th power \textit{Ruth-Aaron} number, we can always find $n>n_0$ that is an $r-$th power \textit{Ruth-Aaron} number, where $0<r<r_0$. Now, because $g_n(r)$ increases on $r>0$, it suffices to show that we can always find $n>n_0$ with $f_{r_0}(n+1)>f_{r_0}(n)$. We consider the function $x^r/\log x$. If we take our $n$ to be greater than $e^{1/r_0}$, then the function $x^{r_0}/\log x$ will be increasing. We thus have
\begin{align}
   \frac{P(n)^{r_0}}{\log P(n)}\ &<\ \frac{n^{r_0}}{\log n}
\end{align}
holds. Meanwhile, because $\frac{p-1}{2}$ and $\frac{p+1}{2}$ are consecutive integers, when they are both greater than $8$, at most one of them is of the form $2^a\cdot 3^b$, and the number of $n\leq x$ with $n=2^a\cdot 3^b$ is at most $O((\log x)^2)$ (Lemma \ref{lem:applyingfnpn}), which means we can apply Lemma \ref{lem:generalizedfnpn} to most $n$.  Thus,
\begin{align}
    f_{r_0}(n+1)-f_{r_0}(n)\ &=\ 2\cdot p^{r_0}-2\cdot 2^{r_0}-f_{r_0}\left(\frac{p-1}{2}\right)-f_{r_0}\left(\frac{p+1}{2}\right)\nonumber\\
    &>\ 2\cdot p^{r_0}-2\cdot 2^{r_0}-P\left(\frac{p-1}{2}\right)^{r_0}\log\left(\frac{p-1}{2}\right)/\log P\left(\frac{p-1}{2}\right)\nonumber\\
    &-\ P\left(\frac{p+1}{2}\right)^{r_0}\log\left(\frac{p+1}{2}\right)/\log P\left(\frac{p+1}{2}\right)\ (\textrm{Lemma \ref{lem:generalizedfnpn}})\nonumber\\
    &>\ 2\cdot p^{r_0}-2\cdot 2^{r_0}-\left(\frac{p-1}{2}\right)^{r_0}-\left(\frac{p+1}{2}\right)^{r_0}\nonumber\\
    &>\ 2\cdot\left( p^{r_0}- 2^{r_0}-\left(\frac{p+1}{2}\right)^{r_0}\right)\nonumber\\
    &>\ 0
\end{align}
when $p$ is sufficiently large. Therefore, when $n=p^2-1$ is sufficiently large, $f_{r_0}(n+1)>f_{r_0}(n)$, which means the \textit{Ruth-Aaron} $r$ with $f_r(n+1)=f_r(n)$ is less than $r_0$ and that we can always construct an infinite decreasing series of $r$, completing our proof.
\end{proof}

\section{Future Work}\label{sec:fut}
There are a few directions for future work on the \textit{Ruth-Aaron} function. First of all, although we improved Erd\Horig{o}s and Pomerance's result \cite{ErPom} by a factor of $(\log\log x)^3(\log\log\log x)/\log x$, it is still far from the actual number of $r-$th power \textit{Ruth-Aaron} numbers up to $x$, as they appear to be extremely rare. Therefore, further research might work towards tightening the existing bounds, as in many cases our estimation is relatively loose. Second, future work might consider placing the \textit{Ruth-Aaron} function in a linear equation. Inspired by Luca and Stănică's result \cite{LuSta} on the number of integers up to $x$ with $\varphi(n)=\varphi(n-1)+\varphi(n-2)$ (See Appendix \ref{app:arith}), we decide to apply the Fibonacci equation to \textit{Ruth-Aaron} numbers.
\begin{defn}
A \textit{Rabonacci} number $n$ is an integer which satisfies 
\begin{align}
    f(n)\ &=\ f(n-1)+f(n-2).
\end{align}
Similarly, a $r-$th power \textit{Rabonacci} number $n$ satisfies
\begin{align}
    f_r(n)\ &=\ f_r(n-1)+f_r(n-2).
\end{align}
\end{defn}
It's obvious that $\mathcal{R}_{(1,1,-1)}(x)+2$ (where $+2$ indicates adding $2$ to every element in the set $\mathcal{R}_{(1,1,-1)}(x)$) is the same as the set of \textit{Rabonacci} numbers not exceeding $x$. Meanwhile, it seems that \textit{Rabonacci} numbers might be even rarer than \textit{Ruth-Aaron} numbers, as there are $42$ \textit{Rabonacci} numbers below $10^6$, in contrast to $149$ \textit{Ruth-Aaron} numbers within the same range. Using an approach similar to the proof of Theorem \ref{thm:oxlnx1me},  we are able to present a partial result on the upper bound on the rabonacci numbers.
\begin{proposition}
When $P(n)\leq x^{1/3}$ and $f(n)>f(n-1)\geq f(n-2)$, the number of \textit{Rabonacci} numbers up to $x$ is at most
\begin{align}
    \#\mathcal{R}_{(1,1,-1)}(x)\ &=\ O\left(\frac{x(\log\log x\log\log\log x)^2}{(\log x)^2}\right).
\end{align}
\end{proposition}
Moreover, Erd\Horig{o}s has conjectured the following regarding the infinitude of \textit{Ruth-Aaron} numbers.
\begin{conj}[Erd\Horig{o}s]
There are infinitely many $n$ with
\begin{align}
    f(n)\ &=\ f(n+1).
\end{align}
\end{conj}
Meanwhile, inspired by his conjecture on the Euler Totient Function \cite{Er}, we have
\begin{conj}
There exists, for every $k\geq 1$, consecutive integers $n,n+1,\dots,n+k$ such that
\begin{align}\label{ktuple2}
    f(n)\ &=\ f(n+1)\ =\ \cdots\ =\ f(n+k).
\end{align}
\end{conj}
So far, we are able to confirm that there is at least one $n$ for $k=1$ and $k=2$. In fact, as of now, two integers are found to be \textit{Ruth-Aaron} numbers when $k=2$. Due to the rarity of \textit{Ruth-Aaron} triples, we conjecture that
\begin{conj}
There are finitely many integers $n$ with
\begin{align}
    f(n)\ &=\ f(n+1)\ =\ f(n+2).
\end{align}
\end{conj}
It's obvious that in \eqref{ktuple2}, as $k$ increases, the number of $n$ satisfying the equation decreases significantly. Therefore, if we could prove the conjecture for $k=2$, then for all $k\geq 2$, the equation $f(n)=f(n+1)=\cdots=f(n+k)$ has only finitely many solutions.\\
\indent Clearly we are unable to prove the conjecture at this point; however, considering the scarcity of \textit{Ruth-Aaron} pairs, we can show that the conjecture hold true if we allow a better result of $\#\mathcal{R}_{(1,-1)}(x)$ for the estimation of the number of integers up to $x$ with $f(n)=f(n+1)=f(n+2)$. In fact, if $\#\mathcal{R}_{(1,-1)}(x)\leq O\left({x^{(1-\epsilon)/2}}\right)$, for any $0<\epsilon<1$, and \textit{Ruth-Aaron} numbers are uniformly distributed, then standard probabilistic models predict that the number of \textit{Ruth-Aaron} triples (and hence quadruples and higher) is finite. We could create a more sophisticated model that takes into account the decay in the density of \textit{Ruth-Aaron} numbers, which will also give a finite bound on the number of triples; we prefer to do the simple model below to highlight the idea. 

\begin{proposition}
If $\#\mathcal{R}_{(1,-1)}(x)\leq O\left({x^{(1-\epsilon)/2}}\right)$, for any $0<\epsilon<1$, and \textit{Ruth-Aaron} numbers are uniformly distributed, then the number of \textit{Ruth-Aaron} triples is $O(1)$.
\end{proposition}

\begin{proof}
First, we consider the probability that integer $n$ is a \textit{Ruth-Aaron} number when $2^x\leq n< 2^{x+1}$. Because we are assuming a uniform distribution, the probability is
\begin{align}
    O\left(\frac{2^{x(1-\epsilon)/2}}{2^x}\right)\ &=\ O\left(2^{-(1+\epsilon) x/2}\right). 
\end{align}
Then the probability of two \textit{Ruth-Aaron} numbers being consecutive is $O\left(2^{-(1+\epsilon x)}\right)$. Since there are $2^x-2$ consecutive integer triples in $[2^x,2^{x+1})$, the expected number of \textit{Ruth-Aaron} triples between $2^x$ and $2^{x+1}$ is at most
\begin{align}
 O\left( (2^x-2)\cdot 2^{-(1+\epsilon)x}\right)\ &\ll\ O\left(2^x\cdot2^{-(1+\epsilon)x}\right)\nonumber\\
 &=\ O\left(\frac{1}{2^{\epsilon x}}\right).
\end{align}
Therefore, the total number of integers with $f(n)=f(n+1)=f(n+2)$ is at most
\begin{align}
    O\left(\sum_{x=1}^\infty\frac{1}{2^{\epsilon x}}\right)\ &\ll\ O(1),
\end{align}
which means there are finitely many \textit{Ruth-Aaron} triples if $\#\mathcal{R}_{(1,-1)}(x)$ has a rate of growth of $O\left(x^{(1-\epsilon)/2}\right)$ or lower, and \textit{Ruth-Aaron} numbers are uniformly distributed.
\end{proof}

\newpage
\appendix
\addtocontents{toc}{\protect\setcounter{tocdepth}{1}}
\section{List of Other Arithmetic Functions with Similar Results}
\label{app:arith}
\subsection{Euler Totient Function}
Erd\Horig{o}s, Pomerance, and Sárközy \cite{ErPomSa} proved that the number of $n$ up to $x$ with
\begin{align}
    \varphi(n)\ &=\ \varphi(n+1),
\end{align}
where $\varphi(x)$ is the Euler Totient function, or the number of integers up to $x$ that are relatively prime to $x$, is at most $x/\exp\{(\log x)^{1/3}\}$ \cite{ErPomSa}. A similar result holds for the Sigma function $\sigma(x)=\sum_{d|n}d$. In the meantime, they conjectured that for every $\epsilon>0$ and $x>x_0(\epsilon)$, the equations $\varphi(n)=\varphi(n+1)$ and $\sigma(n)=\sigma(n+1)$ each have at least $x^{1-\epsilon}$ solutions up to $x$ \cite{ErPomSa}. \\
\indent Meanwhile, Luca and Stănică  \cite{LuSta} proved that the number of $n$ up to $x$ with $\phi(n)=\phi(n-1)+\phi(n-2)$ is
\begin{align}
    O\left(C(t)\cdot\frac{x\log\log\log x}{\sqrt{\log\log x}}\right),
\end{align}
where $1\leq t<\exp(\log x/\log\log x)$ and $C(t)$ is a constant dependent upon $t$. Such $n$ are referred to as \textit{Phibonacci} numbers \cite{Ba}.

\subsection{Prime Omega Function} Let the unique prime factorization of $n$ be $n=p_1^{a_1}p_2^{a_2}\cdots p_k^{n_k}$, then the Prime Omega function $\omega(n)=k$.
Erd\Horig{o}s et al. conjectured that the number of $n\leq x$ with
\begin{align}
    \omega(n)\ &=\ \omega(n+1)
\end{align}
is $O\left(\frac{x}{\sqrt{\log\log x}}\right)$ \cite{ErPomSa}.

\section{Important Theorems and Conjectures}\label{app:important}
\subsection{The Prime Number Theorem}
The Prime Number Theorem gives an asymptotic form for the prime counting function $\pi(x)$. In 1795, the 15-year-old Gauss conjectured that
\begin{align}
    \pi(x)\ &\sim\ \frac{x}{\log x}.
\end{align}
He later refined his estimate to
\begin{align}
    \pi(x)\ &\sim\ \textrm{Li}(x),
\end{align}
where $\textrm{Li}(x)=\int_2^x\frac{x}{\ln x}dx$ is the logarithmic integral. This result has then been refined numerous times up to this day, involving other arithmetic functions as well as modification on the error term \cite{Wei}. In this paper, we apply the weaker version of the Prime Number Theorem:
\begin{align}
    \pi(x)\ &=\ \frac{x}{\log x}+O\left(\frac{x}{(\log x)^2}\right)\nonumber\\
    &=\ O\left(\frac{x}{\log x}\right)
\end{align}
in order to compute the number of $n$ up to $x$ under restrictions on the largest prime factor of $n$, $n+1$, and $n/P(n)$.

\subsection{On the Largest Prime Factor}\label{bruijnresult}
De Bruijn \cite{Bru} introduces the upper bound on the number of $n\leq x$ with prime factors no greater than $y$, denoted as $\Psi(x,y)$:
\begin{align}
    \Psi(x,y)\ &<\ x\exp\left\{-\frac{\log\log\log y}{\log y}\log x+\log\log y+\log\log y +O\left(\frac{\log\log y}{\log\log\log y}\right)\right\}.
\end{align}
He then refined the result to, for $2<y\leq x$,
\begin{align}\label{bruijnmostrefined}
    \log\Psi(x,y)\ &=\ Z\left\{1+O\left(\frac{1}{\log x}\right)+O\left(\frac{1}{\log\log x}\right)+O\left(\frac{1}{u+1}\right)\right\},
\end{align}
where $u=\log x/\log y$ and
\begin{align}
    Z\ &=\ \left(\log \left(1+\frac{\log y}{\log x}\right)\right)\cdot\frac{\log x}{\log y}+\left(\log\left(1+\frac{\log x}{y}\right)\right)\cdot\frac{y}{\log y}.
\end{align}
As stated in Lemma \ref{lem:lowerboundofp}, De Bruijn's result provides us with a crucial lower bound on the largest prime factor of $P(n)$, which we then used extensively throughout the paper.

\subsection{The Chinese Remainder Theorem}
\begin{thm}[The Chinese Remainder Theorem]
If $m_1,m_2,\dots,m_k$ are pairwise relatively prime positive integers, and if $a_1,a_2,\dots,a_k$ are any integers, then the simultaneous congruences
\begin{align}
    x\ &\equiv\ a_1\ (\textrm{mod }m_1)\nonumber\\
    x\ &\equiv\ a_2\ (\textrm{mod }m_2)\nonumber\\
    &\cdots\nonumber\\
    x\ &\equiv\ a_k\ (\textrm{mod }m_k)
\end{align}
have solution that is unique modulo the product $m_1m_2\cdots m_k$ \cite{DiPeiSal}. In our case, for $p=P(n),q=P(n+1)$, we have
\begin{align}
    n\ &\equiv\ 0\ (\textrm{mod }p)\nonumber\\
    n\ &\equiv\ -1\ (\textrm{mod }q).
\end{align}
\end{thm}
Applying the Chinese Remainder Theorem, we have the number of $n$ up to $x$ is at most $1+\left\lfloor\frac{x}{pq}\right\rfloor$. This result is used extensively in the paper.

\subsection{Mersenne Primes}
We first considered using the conjectured infinitude of Mersenne primes, or primes of the form $2^n-1$, to show the infinitude of $r-$th power \textit{Ruth-Aaron} numbers for real $r$. In fact, Lenstra, Pomerance, and Wagstaff \cite{Wag} have conjectured the infinitude of Mersenne primes, and they also proposed that the number of Mersenne primes less than $x$ is approximately $e^{\gamma}\log_2\log_2(x)$, where $\gamma$ is the Euler-Mascheroni constant. Mathematicians have been in search of Mersenne primes both manually and computationally. As of now, the first $51$ Mersenne primes have been found, with the largest being $2^{8258933}-1$, which was discovered in December 2018 and is the largest publicly prime known.




\ \\


\begin{thebibliography}{999999999} 

\bibitem[Al]{Al} T. Albu, \emph{The irrationality of sums of radicals via Cogalois theory}, Analele Stiintifice ale Universitatii Ovidius Constanta, Seria Matematica. 19, 2011.

\bibitem[Ba]{Ba} A. Bager, \emph{Proposed Problem AMME 2833}, Am. Math. Monthly 87, 404, 1980; with a solution in \emph{ibid.} 88, 622, 1981.

\bibitem[B\'e]{Be} \'E. B\'ezout, \emph{Th\'eorie g\'en\'erale des \'equations alg\'ebriques}. Paris, France: Ph.-D. Pierres, 1779.

\bibitem[Bo]{Bo} Harvard College. The Harvard College mathematics review. Cambridge, MA: Harvard College, 2007. \bburl{http://abel.harvard.edu/hcmr/issues/2.pdf}.

\bibitem[Bru]{Bru} N. G. de Bruijn, \emph{On the number of positive integers $\leq x$ and free of prime factors $>y$. II}. Proceedings of the Koninklijke Nederlandse Akademie van Wetenschappen: Series A: Mathematical Sciences, 69(3), 239-247, 1996.

\bibitem[CCEKLM]{CCEKLM} P. Cohen, K. Cordwell, A. Epstein, C. H. Kwan, A. Lott, \& S. J. Miller, \emph{On Near Perfect Numbers}, 2016.

\bibitem[Di]{Di}
K. Dickman, \emph{On the density of the abundant numbers containing prime factors of a certain relative magnitude}, Ark Mal., Astronomi och Physik 22, no. A10, 1930.

\bibitem[DiPeiSal]{DiPeiSal}
C. Ding, D. Pei, and A. Salomaa, \emph{Chinese Remainder Theorem: Applications in Computing, Coding, Cryptography}, World Scientific Publishing Co., Inc., USA, 1996.

\bibitem[Ell]{Ell} P. D. T. A. Elliott \emph{Probabilistic number theory (Vol. II)}, New York: Springer, 1980.

\bibitem[Er]{Er}
P. Erd\Horig{o}s, \emph{Some remarks on Euler's $\phi$ function and some related problems}. Bull. Amer. Math. Soc. 51, no. 8, 540-544, 1945.

\bibitem[ErPom]{ErPom} P. Erd\Horig{o}s, C. Pomerance, \emph{On the largest prime factors of $n$ and $n + 1$}, Aequat. Math. 17, 115, 1978.

\bibitem[ErPomSá]{ErPomSa}P. Erd\Horig{o}s, C. Pomerance, and A. Sárközy, \emph{On locally repeated values of certain arithmetic functions. II}, Acta Math Hung 49, 251–259, 1987.

\bibitem[ForKon]{ForKon}
K. Ford and S. Konyagin. \emph{On two conjectures of Sierpinski concerning the arithmetic functions $\sigma$ and $\phi$}, 2019.

\bibitem[For]{For}
K. Ford, \emph{Solutions of $\phi(n)=\phi(n+k)$ and $\sigma(n)=\sigma(n+k)$}, 2020.

\bibitem[GraHoPom]{GraHoPom} S. W. Graham, J. J. Holt, and C. Pomerance, \emph{On the solutions to $\phi(n) = \phi(n + k)$}, 1999.

\bibitem[HalRi]{HalRi} H. Halberstam, H. E. Richert, \emph{Sieve Methods}, Academic Press, London, 1974.

\bibitem[Lan]{Lan} E. Landau, \emph{Ueber die zahlentheoretische Funktion $ \phi (n)$ und ihre Beziehung zum Goldbachschen Satz}, Nachr. K\"oni. Ges. Wiss. G\"ottingen, math.-phys. KI. 177-186, 1900.


\bibitem[LuStă]{LuSta}
F. Luca and P. Stănică ,  \emph{Linear equations with the Euler totient function}, Acta Arithmetica 128, 135-147, 2007.

\bibitem[LüWang]{LuWang}
X. Lü and Z. Wang, \emph{On the largest prime factors of consecutive integers}, 2018.

\bibitem[KonIv]{KonIv} J.M. De Koninck, A. Ivić, \emph{The distribution of the average prime divisor of an integer}, Arch. Math 43, 37–43, 1984.

\bibitem[Mi]{Mi} S. J. Miller, \emph{The Probability Lifesaver}, Princeton University Press, Princeton, NJ, 2018. \bburl{https://web.williams.edu/Mathematics/sjmiller/public_html/probabilitylifesaver/index.htm}.

\bibitem[Mont]{Mont} H. L. Montgomery, R. C. Vaughan, \emph{The large sieve}. Mathematika, 20(2), 119-134, 1973.

\bibitem[NePePo]{NePePo} C. Nelson, D. E. Penney, and C. Pomerance, \emph{$714$ and $715$}. J. Recreational Math. 7, 87-89, 1974.

\bibitem[Nom]{Nom} 
\emph{Nombres - curiosités, théorie et usages}, \bburl{http://villemin.gerard.free.fr/aNombre/TYPDIVIS/RuthAaro.htm}

\bibitem[OEIS]{OEIS}
OEIS, \emph{Sequence A001366},  \bburl{https://oeis.org/A001366}.

\bibitem[Pom]{Pom}
C. Pomerance, \emph{Ruth-Aaron numbers revisited}, Paul Erd\"os and His Mathematics, I (Series: Bolyai Soc. Math.
Stud., Vol. 11) (G. Halasz, L. Lovasz, M. Simonovits, and V. Sos, eds.), Springer-Verlag, New York 567–579, 2002.

\bibitem[Ran]{Ran} R. A. Rankin, \emph{The Difference Between Consecutive Prime Numbers. IV}. \emph{Proceedings of the American Mathematical Society}, vol. 1, no. 2, 143–150, 1950.

\bibitem[Rib]{Rib}
P. Ribenboim \emph{Catalan Conjecture}, Academic Press, Boston, 1994.

\bibitem[Sha]{Sha} C. R. Shalizi, \emph{Advanced Data Analysis from an Elementary Point of View}, Carnegie Mellon University, Appendix A.

\bibitem[Wag]{Wag}
S. S. Wagstaff, Jr. \emph{Divisors of Mersenne numbers}, Math. Comp. 40 no.161, 385–397, 1983.

\bibitem[Wei]{Wei} E. W. Weisstein, ``Cauchy's Inequality.'' From MathWorld---A Wolfram Web Resource. \bburl{https://mathworld.wolfram.com/CauchysInequality.html}\\
--- --- ---, ``Jensen's Inequality.'' From MathWorld---A Wolfram Web Resource.
\bburl{https://mathworld.wolfram.com/JensensInequality.html}\\
--- --- ---, ``Linearly Independent.'' From MathWorld---A Wolfram Web Resource. \bburl{https://mathworld.wolfram.com/LinearlyIndependent.html}\\
--- --- ---,  ``Prime Number Theorem.'' From MathWorld--A Wolfram Web Resource. \bburl{https://mathworld.wolfram.com/PrimeNumberTheorem.html}\\

\end{thebibliography}
\end{document}